\documentclass[12pt]{amsart}
\usepackage{amsmath}
\usepackage{amssymb}
\usepackage{amsfonts}
\usepackage{color}

\theoremstyle{plain}

\newtheorem{corollary}{Corollary}

\newtheorem{definition}{Definition}
\newtheorem{example}{Example}

\newtheorem{lemma}{Lemma}

\newtheorem{proposition}{Proposition}
\newtheorem{remark}{Remark}

\newtheorem{theorem}{Theorem}
\numberwithin{equation}{section}

\begin{document}
\title[Gabor analysis on the flat cylinder]{Gabor frames for quasi-periodic
functions and polyanalytic spaces on the flat cylinder}
\dedicatory{Dedicated to Karlheinz Gr\"{o}chenig on the occasion of his 65th
birthday}
\author{L. D. Abreu}
\address{Faculty of Mathematics \\
University of Vienna \\
Oskar-Morgenstern-Platz 1 \\
1090 Vienna, Austria}
\email{abreuluisdaniel@gmail.com}
\author{F. Luef }
\address{Department of Mathematics, NTNU Trondheim, 7041 Trondheim, Norway}
\email{franz.luef@ntnu.no}
\author{M. Ziyat}
\email{ziyatmohammed17@gmail.com}
\address{Laboratory of Analysis and Applications, Department of Mathematics,
Faculty of Sciences, Mohammed-V University in Rabat, P.O. Box 1014, Rabat,
Morocco}
\thanks{The authors would like to thank Antti Haimi for his input during the
early stages of this work, to Allal Ghanmi for motivating discussions and to
Michael Speckbacher for valuable remarks. This research was supported by the
Austrian Science Fund (FWF), P-31225-N32 and 10.55776/PAT8205923.}

\begin{abstract}
We develop an alternative approach to the study of Fourier series, based on
the Short-Time-Fourier Transform (STFT) acting on $L_{\nu }^{2}(0,1)$, the
space of measurable functions $f$ in $\mathbb{R}$, square-integrable in $%
(0,1)$, and \emph{time-periodic up to a phase factor}: for fixed $\nu \in 
\mathbb{R}$, 
\begin{equation*}
f(t+k)=e^{2\pi ik\nu }f(t)\text{, \ }k\in \mathbb{Z}\text{.}
\end{equation*}%
The resulting phase space is the vertical strip $\mathbb{C}/\mathbb{Z}%
=[0,1)\times \mathbb{R}$, a flat model of an infinite cylinder, which leads
to Gabor frames with an interesting structure theory, allowing a
Janssen-type representation. As expected, a Gaussian window leads to a Fock
space of entire functions, studied in the companion paper by the same
authors [\emph{Beurling-type density theorems for sampling and interpolation
on the flat cylinder}]. When $g$ is a Hermite function, we are lead to \emph{%
true} Fock spaces of polyanalytic functions (\emph{Landau Level}
eigenspaces) on the vertical strip $[0,1)\times \mathbb{R}$. \ We first
prove a density conditions for interpolating in this space. Furthermore, an
analogue of the sufficient Wexler-Raz conditions is obtained. This leads to
a new criteria for Gabor frames in $L^{2}(\mathbb{R})$, to sufficient
conditions for Gabor frames in $L_{\nu }^{2}(0,1)$ with Hermite windows (an
analogue of a theorem of Gr\"{o}chenig and Lyubarskii about Gabor frames
with Hermite windows) and with totally positive windows in the Feichtinger
algebra (an analogue of a recent theorem of Gr\"{o}chenig). We also consider
a vectorial STFT in $L_{\nu }^{2}(0,1)$ and, using the vector with the first
Hermite functions as window, we introduce the (full) Fock spaces of
polyanalytic functions on $[0,1)\times \mathbb{R}$, their associated
Bargmann-type transforms, and prove an analogue of Vasilevski's orthogonal
decomposition into true polyanalytic Fock spaces (Landau level eigenspaces
on $[0,1)\times \mathbb{R}$). We conclude the paper with an analogue of Gr%
\"{o}chenig-Lyubarskii's sufficient condition for Gabor super-frames with
Hermite functions, which is equivalent to a sufficient sampling condition on
the full Fock space of polyanalytic functions on $[0,1)\times \mathbb{R}$.
The proofs of the results about Gabor frames, involving some of Gr\"{o}%
chenig's most significant results of the last 20 years, provide a live
indicator of his influence on the field during this period.
\end{abstract}

\maketitle

\section{Introduction}

The theory to be developed in this paper can be understood as a study of the
properties of time-frequency representations of classical Fourier series.
Given $\nu \in \mathbb{R}$,\ the functions that we will consider on the
`signal side' can be written as the Fourier series 
\begin{equation*}
f(t)=\sum_{k\in \mathbb{Z}}a_{k}e^{2\pi it(\nu +k)}\text{,}
\end{equation*}%
and their STFT transform $V_{g}$ (this is defined in the next section; see
also Proposition \ref{Proposition1} (3)) can be written as the \emph{%
orthonormal expansion}%
\begin{equation}
V_{g}f(x,\xi )=\sum_{k\in \mathbb{Z}}a_{k}e^{2i\pi x(\nu +k-\xi )}\mathcal{F}%
(\bar{g})(\xi -\nu -k)\text{,}  \label{FourierTF}
\end{equation}%
with convergence in $L^{2}([0,1)\times \mathbb{R})$. Since this expansion
lives in a reproducing kernel Hilbert space, the convergence is
automatically uniform, in contrast with Fourier series. Moreover, point
evaluations of $V_{g}f(x,\xi )$ are bounded if $(x,\xi )\in \lbrack
0,1)\times \mathbb{R}$, suggesting the possibility of recovery of $%
V_{g}f(x,\xi )$ from its samples\ on a discrete subset of $(0,1)\times 
\mathbb{R}$ (e. g., the regular lattice $Z=i\beta \mathbb{Z}$). Given $g\in {%
L}^{2}(\mathbb{R})$, one of our main goals is to obtain conditions on the
size of the parameter $\beta >0$ such that the mentioned sampling recovery
can be done in a stable way. This problem will lead to a theory of \emph{%
Gabor frames} adapted to this context.\ For $\nu =0$, (\ref{FourierTF})\ is
exactly the STFT windowed by $g$, of the Fourier series of $f$ in the
time-frequency plane. As we will see in this paper and its companion \cite%
{ALZ}, this leads to time-frequency objects different from those generated
by general transforms in $L^{2}(\mathbb{R})$, but with structural properties
reminiscent of the classical theory, which can be used in the study of the
difficult problem of determining the points which generate Gabor frames with
a given window \cite{Mystery}.

This paper and \cite{ALZ} are essentially dealing with the $L_{2}$ theory of
the STFT on the flat cylinder. We are currently investigating $L_{p}$
aspects of the theory, and will present in \cite{AZLBanach} a detailed study
of the modulation spaces, $p$-Fock spaces and extensions to $L_{p}$ of the
sufficient conditions for frames, obtained in \cite{ALZ} for the Gaussian
case and in section 4 of this paper to other cases, including Hermite
functions.

\subsection{The Short{-time Fourier transform}}

The \emph{time-frequency shifts} of a function $f:{\mathbb{R}^{d}}%
\rightarrow {\mathbb{C}}$ (in the `signal space'), consist of the
composition of the modulation operator, acting on $f$ by $M_{\xi
}f(t)=e^{2\pi i\xi t}f(t)$, with the translation operator, acting on $f$ by $%
T_{x}f(t)=f(t-x)$, resulting in the representation coefficient: 
\begin{equation*}
\pi (x,\xi )f(t):=M_{\xi }T_{x}=e^{2\pi i\xi t}f(t-x),\qquad (x,\xi )\in {%
\mathbb{R}^{2d}},t\in {\mathbb{R}^{d}}\text{,}
\end{equation*}%
which defines a projective representation of ${\mathbb{R}^{2d}\equiv \mathbb{%
C}^{d}}$ acting on ${L}^{2}(\mathbb{R}^{d})$, with co-cycle $\sigma $ given
by $\sigma ((x,\xi ),(x^{\prime },\xi ^{\prime }))=e^{\pi i(x^{\prime }\xi
^{\prime }-x\xi )}$ (here, ${\mathbb{C}^{d}}$ is identified with ${\mathbb{R}%
^{2d}}$ by the correspondence $(x,\xi )\leftrightarrow x+i\xi $). Since the
Heisenberg group $\mathbb{H=}{\mathbb{R}^{2d}}\times \mathbb{R}$\ is a
central extension of ${\mathbb{R}^{2d}}$, $\pi $ can be understood as the ${%
\mathbb{R}^{2d}}$ pull back of Schr\"{o}dinger's representation of $\mathbb{H%
}$ acting on ${L}^{2}(\mathbb{R}^{d})$, defined by \cite[(9.16)]{Charly}: 
\begin{equation*}
\rho (x,\xi ,\tau )g(t)=e^{2\pi i\tau }e^{\pi ix\xi }e^{2\pi it\xi
}g(t-x)=e^{2\pi i\tau }e^{\pi ix\xi }\pi (x,\xi )g(t)\text{.}
\end{equation*}%
Since $\rho $\ is an irreducible unitary representation of $\mathbb{H}$ \cite%
[Theorem 9.2.1]{Charly}, its ${\mathbb{R}^{2d}}$ pull back $\pi $ is also
unitary and irreducible.

Given $f\in L^{2}(\mathbb{R}^{d})$, the{\ \emph{short-time Fourier transform
(STFT)}} with \emph{window} $g\in L^{2}(\mathbb{R}^{d})$ transforms a
function of a variable $t\in {\mathbb{R}^{d}}$\ into a function of two
variables $(x,\xi )\in \mathbb{R}^{2d}$ (time-frequency/phase space: as
physical quantities, $x$ can represent time or position, while $\xi $ can
represent frequency or momentum), defined as {\ 
\begin{equation}
V_{g}f(x,\xi )=\left\langle f,\pi (x,\xi )g\right\rangle _{L^{2}(\mathbb{R}%
^{d})}=\int_{%
\mathbb{R}
^{d}}f(t)\overline{g(t-x)}e^{-2\pi i\xi t}dt\in L^{2}(\mathbb{R}^{2d})\text{.%
}  \label{STFT}
\end{equation}%
A fundamental property of the STFT is the Moyal formula: given }$%
f_{1},f_{2},g_{1},g_{2}\in L^{2}(\mathbb{R}^{d})$,%
\begin{equation}
\left\langle V_{g_{1}}f_{1},V_{g_{2}}f_{2}\right\rangle _{L^{2}(\mathbb{R}%
^{2d})}=\left\langle f_{1},f_{2}\right\rangle _{L^{2}(\mathbb{R}%
^{d})}\left\langle g_{1},g_{2}\right\rangle _{L^{2}(\mathbb{R}^{d})}\text{.}
\label{Moyal_STFT}
\end{equation}%
Some important conventions will be used all over the paper. {We will always
assume }$\Vert g\Vert _{L^{2}(\mathbb{R}^{d})}^{2}=1$; this, combined {with
Moyal's formula (\ref{Moyal_STFT}), assures that the} map 
\begin{equation}
V_{g}:L^{2}(\mathbb{R}^{d})\rightarrow L^{2}(\mathbb{R}^{2d})
\label{STFTmap}
\end{equation}%
is an isometry. For each $g\in L^{2}(\mathbb{R}^{d})$, the STFT map (\ref%
{STFTmap}) defines a proper subspace of the space $L^{2}(\mathbb{R}^{2d})$.
This space is the so-called \emph{Gabor space }\cite{AbrGabor}, for which
the following notation will be used:%
\begin{equation}
\mathcal{H}_{g}\left( \mathbb{R}^{2d}\right) :=V_{g}\left( L^{2}(\mathbb{R}%
^{d})\right) \text{.}  \label{GaborSpaceC}
\end{equation}%
The Gabor space\ is a Hilbert space {with reproducing kernel }explicitly
given as{\ } 
\begin{equation}
K_{g}((x,\xi ),(x^{\prime },\xi ^{\prime }))=\left\langle \pi (x,\xi )g,\pi
(x^{\prime },\xi ^{\prime })g\right\rangle _{L^{2}(\mathbb{R}^{d})}=e^{-2\pi
ix^{\prime }(\xi -\xi ^{\prime })}V_{g}g\left( (x-x^{\prime },\xi -\xi
^{\prime })\right) \text{.}  \label{kernel}
\end{equation}%
For instance, if $g(t)=h_{0}(t)=2^{\frac{d}{4}}e^{-\pi t^{2}}$, then (${%
\mathbb{C}^{d}}$ is identified with ${\mathbb{R}^{2d}}$ by the
correspondence $(x,\xi )\leftrightarrow x+i\xi =z$):%
\begin{equation}
K_{h_{0}}(z,z^{\prime })=e^{i\pi (x\xi -x^{\prime }\xi ^{\prime })}e^{-\frac{%
\pi }{2}(\left\vert z\right\vert ^{2}+\left\vert z^{\prime }\right\vert
^{2})}e^{\pi z\overline{z^{\prime }}}\text{.}  \label{kernelGaussian}
\end{equation}

To simplify the exposition, in the remaining of the paper we will set $d=1$.
In most algebraic arguments we can do this with no loss of generality, but
once analysis enters the picture, things can become quite complicated in
higher dimensions. For instance, choosing the window $g(t)=h_{0}(t)$, one
can recognize the kernel (\ref{kernelGaussian}) as a weighted version of the
reproducing kernel of the Fock space of entire functions, whose \emph{zeros
have a point set structure only for }$d=1$. Exploring this point structure
led to the characterization of sampling and interpolating sequences in Fock
spaces of entire functions \cite{seip1,seip2,Ly,ALZ}, which can be rephrased
as a characterization of Gabor frames and Riesz sequences for $L^{2}(\mathbb{%
R})$. For $d>1$, zeros of entire functions have an elusive structure, which
may include both points and manifolds, and the corresponding
sampling/interpolation problem becomes highly non-trivial, with results
concerning hypersurfaces as sampling nodes obtained in \cite{FockHyper},
with the first results on lattices obtained in \cite{Groe2011,Multivariate}
and a general condition obtained in \cite{LueWang}, after rephrasing the
notion of Beurling density in terms of the Seshadri constant. A recent
preprint by Testorf \cite{JohannesTorus} offers a convincing extension of
the results on flat tori \cite[Theorem 3, Theorem 14]{TFTori} to higher
dimensions, also in terms of the Seshadri constant.

\subsection{Overview and context}

Given a discrete subgroup $\Gamma $ {of }$\left( \mathbb{R}^{2},+\right) $
(a \emph{lattice}),{\ }one can identify the topological space $\mathbb{R}%
^{2}/\Gamma $\ with the fundamental domain of the lattice, a set $\Lambda
(\Gamma )$ tilling $\mathbb{R}^{2}$ by $\Gamma $-translations. $\Lambda
(\Gamma )$ is compact iff the lattice has full rank $2$. We will develop a
theory where $\Lambda (\Gamma )$\emph{\ becomes a phase space of the STFT
acting on a signal space where }$\Gamma $\emph{-periodic conditions have
been imposed}. This contains three fundamentally different cases:

\begin{enumerate}
\item The \emph{classical} trivial case $\Gamma =\{0\}$. In this case, the
signal space is a space of functions $f:{\mathbb{R}}\rightarrow {\mathbb{C}}$%
, the phase space is $\mathbb{R}^{2}/\{0\}=\mathbb{R}^{2}$ and no
periodicity conditions are imposed. It is the classical setting defined in
Section 1.1 above.

\item The \emph{co-compact} case, where $\Gamma $ is an Euclidean lattice of
full rank 2. Then $\mathbb{R}^{2}/\Gamma $ are compact surfaces (tori),
topologically equivalent to the fundamental domains of the lattices $\Gamma $
, named as \emph{flat tori}. For instance, if $\Gamma _{N}=\mathbb{Z\times NZ%
}$, then $\Lambda (\Gamma _{N})=[0,1]\times \lbrack 0,N]$ is the associated
flat tori, which becomes a phase space when the STFT acts on a $N$%
-dimensional space generated by sequences of delta trains \cite[Theorem 3,
Theorem 14]{TFTori}. Slightly different Fock-type spaces, with
Frobenius-type periodic conditions (\cite{Frobenius},\cite[Chapter VI]{Lang}%
) have been studied in \cite%
{DimensionTheta,GhanmiIntissar2008,GhanmiIntissar2013}.

\item The \emph{non co-compact }case of the rank one lattice $\Gamma =\omega 
\mathbb{Z}$, where $\omega \in \mathbb{R}\setminus \{0\}$. This leads to the
flat cylinder as the fundamental region representing $\mathbb{R}^{2}/\Gamma $%
. The signal space is a space of functions $f:{\mathbb{R}}\rightarrow {%
\mathbb{C}}$ is the interval $[0,\omega )$ and the phase space the \emph{%
unbounded} vertical strip $\Lambda (\Gamma )=[0,\omega )\times \mathbb{R}$,
which represents $\mathbb{R}^{2}/\Gamma $ as a \emph{flat cylinder}. If we
set $\omega =1$, then $\Lambda (\mathbb{Z})=[0,1)\times \mathbb{R}$ and the
quasi-periodic `signal space' is ${L}_{\nu }^{2}(0,1)$.
\end{enumerate}

In this work we will be concerned with the non-compact case (3) and
investigate the consequences of letting the STFT (\ref{STFT}) act on the
signal space ${L}_{\nu }^{2}(0,1)$ of\ functions $f$, measurable in $\mathbb{%
R}$, square-integrable in $(0,1)$, and almost periodic in the sense that,
for a fixed parameter $\nu \in \mathbb{R}$, 
\begin{equation}
f(t+k)=e^{2\pi ik\nu }f(t)\text{, \ }k\in \mathbb{Z}\text{.}
\label{periodic}
\end{equation}%
The functional equation (\ref{periodic}) is transformed into the
quasi-periodicity conditions (\ref{PhasePeriodic}) for the resulting STFT
transformed space, the Gabor space%
\begin{equation}
\mathcal{H}_{g}^{\nu }\left( \mathbb{R}^{2}/\mathbb{Z}\right) :=V_{g}\left( {%
L}_{\nu }^{2}(0,1)\right) \text{,}  \label{GaborSpace}
\end{equation}%
a proper subspace of the space $L^{2}(\mathbb{R}^{2}/\mathbb{Z})$ of
measurable functions $F$ on $\mathbb{R}^{2}$, satisfying%
\begin{equation}
F((x+k,\xi ))\,=e^{2\pi ik(\nu -\xi )}F(x,\xi )  \label{PhasePeriodic}
\end{equation}%
and square-integrable on the vertical strip 
\begin{equation}
\Lambda (\mathbb{Z}):=\left[ 0,1\right) \times \mathbb{R}\text{.}
\label{strip}
\end{equation}%
Clearly,%
\begin{equation*}
\mathbb{R}^{2}=\bigcup\limits_{k\in \mathbb{Z}}\Lambda (\mathbb{Z})+k\text{.}
\end{equation*}%
Thus, the region $\Lambda (\mathbb{Z})$ is a fundamental domain for the
quotient group $\mathbb{R}^{2}/\mathbb{Z}$, which is the flat model of an
infinite cylinder and, in analogy with the flat torus \cite[Theorem 3,
Theorem 14]{TFTori}, we will call it \emph{the flat cylinder. }In higher
dimensions,\emph{\ }the corresponding geometric objects are know as \emph{%
quasi-tori }\cite{IZ2,Z1,Ziyat} and one expects the possibility of studying
Gabor frames in this setting, via We note in passing that Testorf \cite%
{JohannesTorus} posted a recent preprint with the extension to several
variables of some of the results \cite[Theorem 3, Theorem 14]{TFTori}.

The classification (1)-(3) was the starting point of a systematic study of
the properties of Fock-type spaces of functions with periodicity conditions,
a program carried out by Ahmed Intissar with his students Ghanmi \cite%
{GhanmiIntissar2008,GhanmiIntissar2013} and the third named author of this
paper \cite{IZ1,IZ2,Z1}. In our STFT approach to (3), choices of Gaussian
and Hermite windows will lead, respectively, to analytic and true
polyanalytic Fock spaces (associated with the Landau Levels in $L^{2}(%
\mathbb{R}^{2}/\mathbb{Z})$) and recover the associated true polyanalytic
Bargmann transforms \cite{Z1,IZ1} for the case (3), extending a similar
identification for the case (1) \cite{Abr2010}. The conditions (\ref%
{PhasePeriodic}) become, under the action of the Bargmann transform,
automorphic conditions associated with the action of the Weyl operator
acting by horizontal translations. In this direction, a contribution of our
work is the construction of the (full) Fock space of polyanalytic functions
on $\mathbb{C}/\mathbb{Z\equiv R}^{2}/\mathbb{Z}$, the extension to the case
(3) of Vasilevski's orthogonal decomposition of polyanalytic Fock spaces in
true polyanalytic Fock spaces \cite{VasiFock} and the construction of the
associated polyanalytic Bargmann transform.

In the classical case (1), of the standard STFT acting in $L^{2}(\mathbb{R})$%
, Landau levels spaces, are defined as the (Fock-type) eigenspaces of the
Landau operator $\mathcal{L}_{z,\overline{z}}$, acting on $L^{2}(\mathbb{C}%
,e^{-\left\vert z\right\vert ^{2}})$, where \cite%
{DimensionTheta,GhanmiIntissar2008,GhanmiIntissar2013,IZ1} \cite{IZ1,IZ2,Z1} 
\begin{equation*}
\mathcal{L}_{z,\overline{z}}:=-\partial _{z}\partial _{\overline{z}}+%
\overline{z}\partial _{\overline{z}}\text{.}
\end{equation*}%
The Landau levels eigenspaces can be identified with STFT spaces with
Hermite windows \cite{Abr2010} and have been identified with Vasilevski's 
\cite{VasiFock} true/pure polyanalytic Fock spaces in \cite[Remark 3]%
{Abr2010}, a connection independently discovered and presented in detailed,
rigorous form, by Mouayn \cite{Mouayn}. These discoveries were spearheaded
by the results of Gr\"{o}chenig and Lyubarskii on Gabor frames with Hermite
functions \cite{CharlyYura,CharlyYurasuper}, and led to new perspectives in
the field of polyanalytic function theory \cite{Balk}.

In the co-compact case (2),\ Fock-type spaces were considered in \cite%
{DimensionTheta,GhanmiIntissar2008,GhanmiIntissar2013,IZ1}. Fixing a
dimension given by the natural number $N$, the choice of the full-rank
lattice $\Gamma _{N}=\mathbb{Z\times NZ}$ leads to spaces on the co-compact
domain $\Lambda (\Gamma _{N})=[0,1]\times \lbrack 0,N]$, which are $N$%
-dimensional. The resulting STFT becomes an off-grid extension of the
Discrete Finite Gabor Transform \cite[Theorem 1]{TFTori}, whose dimension $N$
can be chosen a priori, adapting to the signal analytic application at
hand.\ A precise description of points generating finite Gabor frames with
the Gaussian window, $g(t)=h_{0}(t)=2^{\frac{1}{4}}e^{-\pi t^{2}}$, using
the analiticity of the resulting Fock space on $\Lambda (\Gamma _{N})$, has
been obtained in \cite[Theorem 3]{TFTori}.

Back to the non-compact case (3) investigated in this paper, we shall see
that, imposing almost periodic conditions over $\Gamma =\mathbb{Z}$ as in
(3), leads to a theory of the STFT of quasi-periodic functions in the time
(horizontal) variable. The time-frequency phase space becomes the vertical
strip $[0,1)\times \mathbb{R}$, which is a flat model of an infinite
cylinder, representing a fundamental domain of $\mathbb{R}^{2}/\mathbb{Z}$.
As in the case of general real valued $L^{2}(\mathbb{R})$ functions, a
Gaussian window leads to the Fock space of entire functions on $\mathbb{C}/%
\mathbb{Z{\equiv }R}^{2}/\mathbb{Z}$, while Hermite windows lead to
eigenspaces of the Landau operator on $\mathbb{R}^{2}/\mathbb{Z}$ considered
in \cite{GhanmiIntissar2008,GhanmiIntissar2013,IZ1}, which are also the $%
\mathbb{R}^{2}/\mathbb{Z}$ versions of the true/pure polyanalytic Fock
spaces, in the language of \cite{VasiFock}. We will complete this picture by
introducing a new object, \emph{the (full) Fock space of polyanalytic
functions on }$\mathbb{R}^{2}/\mathbb{Z}$, prove the analogue of
Vasilevski's decomposition \cite{VasiFock} of this new space, as an
orthogonal sum of true polyanalytic Fock spaces, and by defining a
Bargmann-type transform which can be related to a vectorial STFT with $n$
Hermite windows. This transform is motivated by applications in multiplexing
(sending several signals simultaneously with a single transmission route 
\cite{balan00,BalanWH}).

This is both a follow-up and a companion of the paper \cite{ALZ}, where the
point sets $Z\subset \Lambda (\mathbb{Z})$, leading to sampling and
interpolation sequences for the (analytic) Fock space on $\mathbb{C}/\mathbb{%
Z}$, are characterized in terms of upper and lower Beurling densities,
revealing as a critical `Nyquist density' the real number $1$, meaning that 
\emph{the condition }$D^{-}(Z)>1$\emph{\ characterizes sets of sampling},
while \emph{the condition }$D^{+}(Z)<1$\emph{\ characterizes sets of
interpolation}. These results have an interpretation in terms of Gabor
frames and Riesz basic sequences. This is equivalent to sampling and
interpolation theorems for the space $\mathcal{H}_{h_{0}}^{\nu }\left( 
\mathbb{R}^{2}/\mathbb{Z}\right) $, where $h_{0}$ is a Gaussian window, and
leads to Gabor frame (resp. Riesz basic sequences) density theorems for the
Gabor system in ${L}_{\nu }^{2}(0,1)$ defined by the parameterized theta
functions $\{e^{2i\pi \xi t}\theta _{\xi ,z_{1}}(-t,i):\quad z=x+i\xi \in
Z\} $.

We will then pursue the study of Gabor frames in ${L}_{\nu }^{2}(0,1)$ for a
general window $g\in L^{2}(\mathbb{R})$, {and obtain, in this setting,
analogues of the so-called Janssen's representation of the Gabor frame
operator and of Wexler-Rax's orthogonality relations. }The Hermite window
choice $g=h_{n}$ leads to Gabor systems of parameterized Hermite-theta
functions of the form 
\begin{equation*}
\theta _{\alpha ,\beta }^{(r)}(z)=\sum_{k\in \mathbb{Z}}e^{-\pi (k+\alpha
)^{2}+2i\pi (z-\beta )(k+\alpha )}h_{r}\left( \sqrt{2\pi }(k+\alpha +\Im
(z)\right)
\end{equation*}%
The Wexler-Rax's relations will allow to show that, derive a sufficient
condition on $\beta $, when $Z$ is a regular lattice $Z=i\beta \mathbb{Z}$, $%
\beta <\frac{1}{r+1}$is a sufficient condition for the Gabor system with
Hermite-theta functions, $\{e^{2i\pi \beta nt}\theta _{\nu -\beta
n,0}^{(r)}(t)\;|\;n\in \mathbb{Z}\}$ to be a Gabor frame for ${L}_{\nu
}^{2}(0,1)$. This extends to our setting well known results of Gr\"{o}chenig
and Lyubarskii for $L^{2}(\mathbb{R})$ \cite{CharlyYura,CharlyYurasuper}.

\subsection{Dedication}

This article is dedicated to Karlheinz `Charly' Gr\"{o}chenig, our
`lighthouse of time-frequency analysis'. After setting a standard for the
mathematical foundations of the discipline with the monograph \cite{Charly},
Gr\"{o}chenig went on a journey across several mathematical fields, often
blurring the frontiers between pure and applied mathematics, ultimately
combining seemling unrelated subjects in the solution of important
mathematical problems. While writing this paper, we were often confronted
with the reality that nowadays, in time-frequency analysis, one cannot go
very far without using some of his results. For a sample, the proofs in our
contribution required, at least: Wiener's lemma for non-commutative tori 
\cite{GL}, ideas from Gabor frames with Hermite functions \cite%
{CharlyYura,CharlyYurasuper} and the recent striking result \cite{GroAIM},
where it is shown that Sch\"{o}enberg's totally positive functions \cite%
{sch51} with separable rational lattices yield frames if and only if $\alpha
\beta <1$.

\subsection{Outline}

The paper is organized as follows. In the next section we present the
elementary properties of quasi-periodic functions and of their STFT. Section
3 considers properties of Fock spaces on $\mathbb{R}^{2}/\mathbb{Z}$, of
analytic and true polyanalytic functions, together with their associated
Bargmann-type transforms. Section 4 is the most significant part of the
paper, where a detailed study of the structure of Gabor frames in ${L}_{\nu
}^{2}(0,1)$ is carried out, leading to an analogue of the Wexler-Raz
relations, which will be used to show several sufficient conditions for
Gabor frames $\mathcal{G}_{\nu }\left( g,Z_{\beta }\right) $ in ${L}_{\nu
}^{2}(0,1)$, where $Z_{\beta }=\beta \mathbb{Z}$. We will see that these
particular frames\ have an intriguing connection with Gabor frames $\mathcal{%
G}\left( g,\Lambda _{1,\beta }\right) $ in $L^{2}\left( \mathbb{R}\right) $,
for the lattice $\Lambda _{1,\beta }=\mathbb{Z}\times \beta \mathbb{Z}$ in $%
\mathbb{R}^{2}$. In the final section, we introduce polyanalytic Fock spaces
in this context, their Bargmann-type transforms and show that they enjoy a
Vasilevski-type orthogonal decomposition into pure polyanalytic Fock spaces
(Landau Levels on $[0,1)\times \mathbb{R}$). The paper is concluded with a
vectorial version of the Wexler-Raz relations for Gabor frames in ${L}_{\nu
}^{2}\left[ (0,1),\mathbb{C}^{N}\right] $, which is used to obtain an
analogue of Gr\"{o}chenig-Lyubarskii's sufficient condition for Gabor
super-frames with Hermite functions \cite{CharlyYurasuper}, or, equivalently
to a sufficient sampling condition on the full Fock space of polyanalytic
functions on $[0,1)\times \mathbb{R}$.

\section{The space ${L}_{\protect\nu }^{2}(0,1)$ and the Short-Time Fourier
Transform}

$\Gamma $-quasi-periodic functions are defined by the functional equation%
\begin{equation*}
f(t+\gamma )=\chi _{\nu }(\gamma )f(t)\text{,}
\end{equation*}%
where $\chi _{\nu }$ is a character from $\Gamma $ to the unit circle,
mapping a generic element $\gamma =n\alpha \in \Gamma $ to $\chi _{\nu
}(\gamma )=e^{2\pi in\nu }$, $\nu \in \mathbb{R}$

\subsection{Almost periodic square integrable functions}

Consider a lattice $\Gamma =\mathbb{Z}\alpha $, $\alpha \in \mathbb{R}%
\setminus \{0\}$, and $\Gamma $-quasi-periodic functions defined by the
functional equation%
\begin{equation}
f(t+\gamma )=\chi _{\nu }(\gamma )f(t)\text{,}  \label{functional}
\end{equation}%
where $\chi _{\nu }$ is a character from $\Gamma $ to the unit circle,
mapping a generic element $\gamma =n\alpha \in \Gamma $ to $\chi _{\nu
}(\gamma )=e^{2\pi ina\nu }$, $\nu \in \mathbb{R}$. Without loss of
generality, we can set $\alpha =1$, then $\Gamma =\mathbb{Z}$ and (\ref%
{functional}) is%
\begin{equation}
f(t+n)=e^{2\pi in\nu }f(t)\text{, \ }n\in \mathbb{Z}\text{.}  \label{shift}
\end{equation}%
The space of square integrable almost periodic functions defines the signals
to be transformed by the STFT: 
\begin{equation*}
{L}_{\nu }^{2}(0,1)=\left\{ f:\mathbb{R}\rightarrow \mathbb{C}\ \text{%
measurable with }\left\Vert f\right\Vert _{L^{2}(0,1)}<\infty \text{,
verifying (\ref{functional})}\right\} \text{.}
\end{equation*}%
From the definition, expanding $f\in L^{2}(0,1)$ in Fourier series, we can
write a generic element $f\in {L}_{\nu }^{2}(0,1)$ as%
\begin{equation}
f(t)=\sum_{k\in \mathbb{Z}}a_{k}e^{2\pi it(\nu +k)}\text{{.}}  \label{signal}
\end{equation}%
It follows that $\{e^{2\pi it(\nu +k)}\}_{k\in \mathbb{Z}}$ is an
orthonormal basis of ${L}_{\nu }^{2}(0,1)$. We will use the notation 
\begin{equation}
e_{k,\nu }(t)=e^{2\pi it(\nu +k)}\text{.}  \label{basis0}
\end{equation}%
Recall that Feichtinger's algebra $\mathcal{S}_{0}(\mathbb{R})$ can be
defined as \cite{Fei,FeiModulation,Modulation,Charly}: 
\begin{equation*}
\mathcal{S}_{0}(\mathbb{R}):=\left\{ g\in {L}^{2}(\mathbb{R}):V_{h_{0}}g\in
L^{1}(\mathbb{R}^{2})\right\} \text{.}
\end{equation*}%
For $f\in \mathcal{S}_{0}(\mathbb{R})$, one defines a periodization operator
by%
\begin{equation}
\Sigma _{\nu }f(t):=\sum_{k\in \mathbb{Z}}e^{2\pi ik\nu }f(t-k)\text{.}
\label{Period}
\end{equation}

\begin{proposition}
\label{Proposition0}$\Sigma _{\nu }$ is well defined as an operator with
dense range,%
\begin{equation*}
\Sigma _{\nu }:\mathcal{S}_{0}(\mathbb{R})\rightarrow {L}_{\nu }^{2}(0,1)%
\text{.}
\end{equation*}
\end{proposition}

\begin{proof}
If $f\in \mathcal{S}_{0}(\mathbb{R})$, by \cite[Proposition 12.1.4]{Charly}
then $f\in L^{1}(\mathbb{R})$ and $f\in W(\mathbb{R})$, where $W(\mathbb{R})$
is the Wiener space such that%
\begin{equation*}
\left\Vert f\right\Vert _{W(\mathbb{R})}=\sum_{k\in \mathbb{Z}}ess\sup_{t\in
\lbrack 0,1]}\left\vert f(t-k)\right\vert <\infty \text{.}
\end{equation*}%
Using \cite[Lemma 6.1.2]{Charly}, if $f\in W(\mathbb{R})$ then 
\begin{equation*}
\left\vert \Sigma _{\nu }f(t)\right\vert \lesssim \sum_{k\in \mathbb{Z}%
}\left\vert f(t-k)\right\vert \lesssim \left\Vert f\right\Vert _{W(\mathbb{R}%
)}\text{.}
\end{equation*}%
It follows that $\Sigma _{\nu }f(t)$ is well defined for $f\in \mathcal{S}%
_{0}(\mathbb{R})$, since, by \cite[Lemma 6.1.2]{Charly}, 
\begin{equation}
\left\vert \Sigma _{\nu }f(t)\right\vert \lesssim \left\Vert f\right\Vert
_{W(\mathbb{R})}\lesssim \left\Vert f\right\Vert _{\mathcal{S}_{0}(\mathbb{R}%
)}\text{.}  \label{bound}
\end{equation}%
Now, if $f\in \mathcal{S}_{0}(\mathbb{R})$, then $\Sigma _{\nu }f$ $\in {L}%
_{\nu }^{2}(0,1)$, since (\ref{bound}) trivially implies that%
\begin{equation}
\int\limits_{\lbrack 0,1]}\left\vert \Sigma _{\nu }f(t)\right\vert
^{2}dt\lesssim \left\Vert f\right\Vert _{\mathcal{S}_{0}(\mathbb{R})}
\label{1}
\end{equation}%
and 
\begin{equation*}
\Sigma _{\nu }f(t+n)=\sum_{k\in \mathbb{Z}}e^{2\pi ik\nu }f(t-k+n)=e^{2\pi
in\nu }\Sigma _{\nu }f(t)\text{.}
\end{equation*}%
We consider the functions 
\begin{equation*}
\varphi _{n}(x)=e^{2i\pi (n+\nu )x}\chi _{\lbrack 0,1]}(x),\quad n\in 
\mathbb{Z}\text{.}
\end{equation*}%
From \cite[Theorem 1 (a)]{Poisson} we obtain the inequality:%
\begin{equation*}
\left\Vert V_{g}f\right\Vert _{1}\leq C\left( \left\Vert (.)^{7}f\right\Vert
_{2}+\left\Vert (.)^{7/4}\widehat{f}\right\Vert _{2}\right) \text{.}
\end{equation*}%
Since $\widehat{\varphi }_{k}(x)=O(1/x)$, we conclude that$\left\Vert
V_{g}\varphi _{k}\right\Vert _{1}<\infty $ and, therefore, $\varphi _{k}\in 
\mathcal{S}_{0}(\mathbb{R})$. Now, 
\begin{equation*}
(\Sigma _{\nu }\varphi _{n})(t)=e^{2i\pi (n+\nu )x}\sum_{k\in \mathbb{Z}%
}\chi _{\lbrack 0,1]}(x-k)=e^{2i\pi (n+\nu )x}\text{.}
\end{equation*}%
Since $\{e^{2\pi it(\nu +n)}\}_{n\in \mathbb{Z}}$ is an orthonormal basis of 
${L}_{\nu }^{2}(0,1)$, then $\Sigma _{\nu }:\mathcal{S}_{0}(\mathbb{R}%
)\rightarrow {L}_{\nu }^{2}(0,1)$ has dense range.
\end{proof}

\begin{remark}
\label{Zak}The operator $\Sigma _{\nu }$ ressembles the Zak transform \cite[%
Chapter 8]{Charly}: 
\begin{equation*}
\mathbf{Z}f(x,\xi ):=\sum_{k\in \mathbb{Z}}f(x-k)e^{2\pi ik\xi }\text{,}
\end{equation*}%
since one can write $\Sigma _{\nu }f(t)=\mathbf{Z}f(t,\nu )$. But there is a
subtle big difference: the Zak transform maps a single variable function to
a function in two variables, $f\rightarrow \mathbf{Z}f(t,\nu )$, while $%
\Sigma _{\nu }$ is an operator with \emph{fixed parameter} $\nu \in \mathbb{R%
}$, mapping a single variable function to another function in a single
variable, $f\rightarrow \Sigma _{\nu }f(t)$.
\end{remark}

\begin{remark}
An alternative approach for the time-frequency analysis of periodic signals 
\cite{GabHan,GabLi,LiZhang,LiangLi}, considers subspaces of $L^{2}(\mathbb{R}%
)$ of the form%
\begin{equation*}
L^{2}(S)=\left\{ f\in L^{2}(\mathbb{R}):f=0\text{ a. e. on }\mathbb{R}%
-S\right\} \text{,}
\end{equation*}%
where $S$ is invariant under $\mathbb{Z}$-shifts.
\end{remark}

\begin{remark}
Yet another approach to the analysis of periodic systems is provided by the
theory of Spline-type spaces \cite{FeiSplyne} (also known as shift-invariant
spaces \cite{RonShiftInva}), which, given a generator $\varphi \in L^{p}(%
\mathbb{R})$, are defined as the closed linear span of the family $%
\{T_{\lambda }\varphi \}_{\lambda \in \Lambda }$ in $L^{p}(\mathbb{R})$,
generated by linear combinations with coefficients $\{a_{\lambda
}\}_{\lambda \in \Lambda }\in l^{p}(\mathbb{R})$.
\end{remark}

\begin{remark}
If we want $\Sigma _{\nu }f$ to satisfy the more general conditions (\ref%
{functional}), where $\chi _{\nu }$ is a character from $\Gamma $ to the
unit circle, mapping a generic element $\gamma =n\alpha \in \Gamma $ to $%
\chi _{\nu }(\gamma )=e^{2\pi ina\nu }$, $\nu \in \mathbb{R}$, we define the
operator $\Sigma _{\nu }:\mathcal{S}_{0}(\mathbb{R})\rightarrow {L}_{\nu
}^{2}(0,a)$ by 
\begin{equation*}
\Sigma _{\nu }f(t):=\sum_{\gamma \in \Gamma }\chi _{\nu }(\gamma )f(t-\gamma
)\text{.}
\end{equation*}
\end{remark}

\begin{remark}
In \cite{TFTori} we have introduced yet another operator, $\mathbf{\Sigma }%
_{N}$, which is reminiscent of $\ \Sigma _{\nu }$,\ but instead of
periodizing the function in time, it periodizes the function simultaneously
in the time and frequency variables:%
\begin{equation*}
\mathbf{\Sigma }_{N}f(t):=\sum_{k_{1},k_{2}\in \mathbb{Z}}e^{2\pi
iNk_{2}t}f(t-k_{1})\text{.}
\end{equation*}%
The operator is fundamental for the treatment of the co-compact case (2) of
the introduction. Notably, the image of $S_{0}(\mathbb{R})$ under $\mathbf{%
\Sigma }_{N}$\ does not contain functions. It consists of\ the space $%
S_{N}\subset S_{0}^{\prime }(\mathbb{R})$ defined as the span of $\{\epsilon
_{n}\}_{n=0}^{N-1}$, a sequence of periodic delta trains, where 
\begin{equation}
\epsilon _{n}:=\sum_{k\in \mathbb{Z}}\delta _{\frac{n}{N}+k}\subset
S_{0}^{\prime }(\mathbb{R}),\qquad n=0,...,N-1\text{.}
\label{def:deltatrain}
\end{equation}
\end{remark}

\subsection{Short-time Fourier transform of quasi-periodic functions}

The main properties of {the action of the STFT as a map}%
\begin{equation*}
V_{g}:{L}_{\nu }^{2}(0,1)\rightarrow \mathcal{H}_{g}^{\nu }\left( \mathbb{R}%
^{2}/\mathbb{Z}\right) \text{,}
\end{equation*}%
where $\mathcal{H}_{g}^{\nu }\left( \mathbb{R}^{2}/\mathbb{Z}\right) ${\ is
the Gabor space (\ref{GaborSpace}), }are collected in the next Proposition.

\begin{proposition}
\label{Proposition1}Let $g\in L^{2}(\mathbb{R})$ with $\Vert g\Vert _{L^{2}(%
\mathbb{R})}^{2}=1$. Then:

\begin{enumerate}
\item For $f\in {L}_{\nu }^{2}(0,1)$, 
\begin{equation}
V_{g}f(x+k,\xi )\,=e^{2\pi ik(\nu -\xi )}V_{g}f(x,\xi )\text{.}
\label{almostperiodic}
\end{equation}

\item For $f_{1},f_{2}\in {L}_{\nu }^{2}(0,1)$ and $g_{1},g_{2}\in L^{2}(%
\mathbb{R})$, 
\begin{equation}
\left\langle V_{g_{1}}f_{1},V_{g_{2}}f_{2}\right\rangle _{L^{2}((0,1)\times 
\mathbb{R})}=\left\langle f_{1},f_{2}\right\rangle _{L^{2}(0,1)}\left\langle
g_{1},g_{2}\right\rangle _{L^{2}(\mathbb{R})}\text{.}  \label{Moyal}
\end{equation}

\item $\mathcal{H}_{g}^{\nu }\left( \mathbb{R}^{2}/\mathbb{Z}\right) $
contains the orthonormal basis: 
\begin{equation}
V_{g}(e_{k,\nu })(x,\xi )=e^{2i\pi x(\nu +k-\xi )}\mathcal{F}(\bar{g})(\xi
-\nu -k)\text{,}\qquad k=0,1,...\text{.}  \label{basis_TF}
\end{equation}

\item Let $g\in \mathcal{S}_{0}(\mathbb{R})$.\ The reproducing kernel of $%
\mathcal{H}_{g}^{\nu }\left( \mathbb{R}^{2}/\mathbb{Z}\right) $ is 
\begin{equation}
K_{g}^{\nu }((x,\xi ),(x^{\prime },\xi ^{\prime }))=e^{-2\pi ix^{\prime
}(\xi -\xi ^{\prime })}\sum_{k\in \mathbb{Z}}e^{2\pi ik(\nu -\xi ^{\prime
})}V_{g}g(x-x^{\prime }-k,\xi -\xi ^{\prime })  \label{repkernel1}
\end{equation}%
or 
\begin{equation}
K_{g}^{\nu }((x,\xi ),(x^{\prime },\xi ^{\prime }))=e^{i\pi (x\xi -x^{\prime
}\xi ^{\prime })}\sum_{k\in \mathbb{Z}}e^{2i\pi (\nu +k)(x-x^{\prime })}%
\mathcal{F}(\bar{g})(\xi -\nu -k)\overline{\mathcal{F}(\bar{g})(\xi ^{\prime
}-\nu -k)}\text{.}  \label{repkernel2}
\end{equation}

\item Let $g,f\in \mathcal{S}_{0}(\mathbb{R})$. Then%
\begin{equation*}
V_{g}(\Sigma _{\nu }f)(x,\xi )=\sum_{k\in \mathbb{Z}}e^{2\pi ik(\nu -\xi
)}V_{g}f(x-k,\xi )\text{.}
\end{equation*}
\end{enumerate}
\end{proposition}

\begin{proof}
(1) If $f\in {L}_{\nu }^{2}(0,1)$, a change of variable gives%
\begin{equation*}
V_{g}f(x+k,\xi )=\int_{%
\mathbb{R}
}f(t+k)\overline{g(t-x)}e^{-2\pi i\xi (t+k)}dt=e^{2\pi ik(\nu -\xi
)}V_{g}f(x,\xi )\text{,}
\end{equation*}%
using (\ref{functional}).

(2) Write $V_{g}f(x,\xi )=\mathcal{F}(f(.)g(.-x))(\xi )$ and assume that $%
f_{1}$ and $f_{2}$ are continuous, so that $f_{j}(.)\bar{g}_{j}(.-x)\in
L^{2}(\mathbb{R})$. Parseval identity then gives 
\begin{eqnarray*}
\int_{(0,1)\times \mathbb{R}}V_{g_{1}}f_{1}(x,\xi )V_{g_{2}}f_{2}(x,\xi
)dxd\xi &=&\int_{\mathbb{R}}\left( g_{1}(t)\bar{g}_{2}(t)%
\int_{(0,1)}f_{1}(t+x)\bar{f}_{2}(t+x)dx\right) dt \\
&=&\int_{\mathbb{R}}g_{1}(t)\bar{g}_{2}(t)dt\int_{(0,1)}f_{1}(x)\bar{f}%
_{2}(x)dx \\
&=&\left\langle g_{1},g_{2}\right\rangle _{L^{2}(\mathbb{R})}\left\langle
f_{1},f_{2}\right\rangle _{L^{2}(0,1)}\text{,}
\end{eqnarray*}%
using the periodicity of $f_{1}(x)\bar{f}_{2}(x)$ on the second line. This
proves (\ref{Moyal}) for continuous functions. For general $f_{1}$,$f_{2}\in 
{L}_{\nu }^{2}(0,1)$, the formula follows from a standard density argument.

(3) Since $\Vert g\Vert _{L^{2}(\mathbb{R})}^{2}=1$ and $\{e_{k,\nu
}(t)\}_{k\in \mathbb{Z}}$ are orthonormal in ${L}^{2}(0,1)$, Moyal's formula
gives 
\begin{equation*}
\left\langle V_{g}e_{k,\nu },V_{g}e_{m,\nu }\right\rangle
_{L^{2}((0,1)\times \mathbb{R})}=\left\langle e_{k,\nu },e_{m,\nu
}\right\rangle _{L^{2}(0,1)}\Vert g\Vert _{L^{2}(\mathbb{R})}^{2}=\delta
_{k,m}\text{.}
\end{equation*}%
Given that $\{e_{k,\nu }(t)\}_{k\in \mathbb{Z}}$ is an orthogonal basis of ${%
L}_{\nu }^{2}(0,1)$, by definition $\{V_{g}(e_{k,\nu })\}_{k\in \mathbb{Z}}$
spans $\mathcal{H}_{g}^{\nu }\left( \mathbb{R}^{2}/\mathbb{Z}\right) $. We
conclude that $\{V_{g}(e_{k,\nu })\}_{k\in \mathbb{Z}}$ provides an
orthonormal basis for $\mathcal{H}_{g}^{\nu }\left( \mathbb{R}^{2}/\mathbb{Z}%
\right) $. (\ref{basis_TF}) follows from direct computation of $V_{g}$
acting on (\ref{basis0}): 
\begin{equation*}
V_{g}(e_{k,\nu })(x,\xi )=\int_{%
\mathbb{R}
}e^{2\pi it(\nu +k-\xi )}\overline{g(t-x)}dt=e^{2i\pi x(\nu +k-\xi )}%
\mathcal{F}(\bar{g})(\xi -\nu -k)\text{.}
\end{equation*}

(4) From (\ref{Moyal}) we easily derive the following inversion formula for
the STFT on ${L}_{\nu }^{2}(0,1)$, valid in a weak sense: 
\begin{equation*}
f=\int_{(0,1)\times \mathbb{R}}V_{g}f(x,\xi )\Sigma _{\nu }\pi (x,\xi
)gdxd\xi \text{.}
\end{equation*}%
Since $g\in \mathcal{S}_{0}(\mathbb{R})$, then by Proposition \ref%
{Proposition0}, $\Sigma _{\nu }g$ and $\Sigma _{\nu }\mathcal{F}(g)$\ are
well-defined. Thus, if $f\in {L}_{\nu }^{2}(0,1)$, then 
\begin{eqnarray*}
V_{g}f(x,\xi ) &=&\int_{(0,1)\times \mathbb{R}}V_{g}f(x^{\prime },\xi
^{\prime })V_{g}\left( \Sigma _{\nu }(\pi (x^{\prime },\xi ^{\prime
})g)\right) (z)dx^{\prime }d\xi ^{\prime } \\
&=&\int_{(0,1)\times \mathbb{R}}V_{g}f(x^{\prime },\xi ^{\prime })\left(
\sum_{k\in \mathbb{Z}}e^{-2\pi ik(\xi -\nu )}K_{g}((x,\xi ),(x^{\prime
}-k,\xi ^{\prime }))\right) dx^{\prime }d\xi ^{\prime }\text{.}
\end{eqnarray*}%
Using the invariance propriety of $K_{g}$ (this can be directly verified
from (\ref{kernel})): 
\begin{equation*}
K_{g}((x+w,\xi ),(x^{\prime }+w,\xi ^{\prime }))=e^{-2i\pi \Re (w)(\xi -\xi
^{\prime })}K_{g}((x,\xi ),(x^{\prime },\xi ^{\prime }))
\end{equation*}%
we conclude that the reproducing kernel of $\mathcal{H}_{g}^{\nu }\left( 
\mathbb{R}^{2}/\mathbb{Z}\right) $ is given by (\ref{repkernel1}). The
alternative form (\ref{repkernel2}) is obtained, after some simplification,
from the representation of the reproducing kernel in terms of the bilinear
sum of its basis functions (\ref{basis_TF}):%
\begin{equation*}
K_{g}^{\nu }((x,\xi ),(x^{\prime },\xi ^{\prime }))=\sum_{k\in \mathbb{Z}%
}e^{2i\pi x(\nu +k-\xi )}\mathcal{F}(\bar{g})(\xi -\nu -k)\overline{e^{2i\pi
x^{\prime }(\nu +k-\xi ^{\prime })}\mathcal{F}(\bar{g})(\xi ^{\prime }-\nu
-k)}\text{.}
\end{equation*}%
(5) First, we observe that, since $f\in \mathcal{S}_{0}(\mathbb{R})$, then
by Proposition \ref{Proposition0}, $\Sigma _{\nu }f\in {L}_{\nu }^{2}(0,1)$
and since $g\in \mathcal{S}_{0}(\mathbb{R})$, $(\Sigma _{\nu }f)(.)\bar{g}%
_{j}(.-x)\in L^{2}(\mathbb{R})$. Then,%
\begin{eqnarray*}
V_{g}(\Sigma _{\nu }f)(x,\xi ) &=&\sum_{k\in \mathbb{Z}}e^{2\pi ik\nu }\int_{%
\mathbb{R}
^{d}}f(t-k)\overline{g(t-x)}e^{-2\pi i\xi t}dt \\
&=&\sum_{k\in \mathbb{Z}}e^{2\pi ik(\nu -\xi )}V_{g}f(x-k,\xi )\text{.}
\end{eqnarray*}
\end{proof}

\subsection{Notational remark}

In the next sections we will often deal with spaces analytic or polyanalytic
in the variable $z=x+i\xi \in \mathbb{C}$. With this in mind, we will
shorten the notations in several formulas, by identifying pairs of
coordinates with complex numbers in the following way: 
\begin{equation*}
z=x+i\xi \equiv (x,\xi )\text{.}
\end{equation*}%
This materializes the identification $\mathbb{C\equiv R}^{2}$ and supports
the notational conventions 
\begin{equation*}
V_{g}f(z)=V_{g}f(x+i\xi )\equiv V_{g}f(x,\xi )\text{,}
\end{equation*}%
\begin{equation*}
K_{g}^{\nu }(z,w)=K_{g}^{\nu }(x+i\xi ,x^{\prime }+i\xi ^{\prime })\equiv
K_{g}^{\nu }((x,\xi ),(x^{\prime },\xi ^{\prime }))
\end{equation*}%
and%
\begin{equation*}
\mathcal{H}_{g}^{\nu }\left( \mathbb{C}/\mathbb{Z}\right) \equiv \mathcal{H}%
_{g}^{\nu }\left( \mathbb{R}^{2}/\mathbb{Z}\right) =V_{g}\left( {L}_{\nu
}^{2}(0,1)\right) \text{.}
\end{equation*}%
With this notations, the horizontal translation of a function $F(z=x+i\xi
)\equiv F(x,\xi )$, by $k\in \mathbb{Z}$, will often be written as%
\begin{equation*}
F(z+k)=F(x+k+i\xi )\equiv F(x+k,\xi )\text{.}
\end{equation*}

\subsection{Examples}

Proposition \ref{Proposition1} can be used for explicit computations of the
basis functions and the reproducing kernels. We will now provide examples
when the window is chosen among the Gaussian and the whole family of Hermite
functions, normalized such that $\Vert h_{r}\Vert _{L^{2}(\mathbb{R})}^{2}=1$%
, explicitly given as 
\begin{equation*}
h_{r}(t)=\frac{2^{1/4-r}(-\pi )^{\frac{r}{2}}}{\sqrt{r!}}e^{\pi
t^{2}}\partial _{r}(e^{-2\pi t^{2}}),\qquad r=0,1,...\text{.}
\end{equation*}%
The examples are related to some of the analytic and polyanalytic Fock
spaces spaces on $\mathbb{C}/\mathbb{Z}$ that will be considered in Section
3. The STFT spaces with Hermite windows $\mathcal{H}_{h_{r}}^{\nu }\left( 
\mathbb{C}\right) $ contain reproducing kernels explicitly given as 
\begin{equation}
K_{h_{r}}(z,w)=e^{-2\pi ix^{\prime }(\xi -\xi ^{\prime
})}V_{h_{r}}h_{r}\left( z-w\right) =e^{i\pi (x\xi -x^{\prime }\xi ^{\prime
})}e^{-\frac{\pi }{2}(\left\vert z\right\vert ^{2}+\left\vert w\right\vert
^{2})}L_{r}(\pi \left\vert z-w\right\vert ^{2})e^{\pi \overline{z}w}\text{,}
\label{kernelHermite}
\end{equation}
where $L_{r}=L_{r}^{(0)}$\ stands for the Laguerre polynomials, defined for
general parameter by the formula%
\begin{equation*}
L_{k}^{\alpha }(x)=\sum\limits_{i=0}^{k}(-1)^{i}\binom{k+\alpha }{k-i}\frac{%
x^{i}}{i!}\text{.}
\end{equation*}

\begin{example}
\label{Example0}Since $h_{0}=2^{\frac{1}{4}}e^{-\pi t^{2}}$ is invariant
under the Fourier transform, Proposition \ref{Proposition1} (3) tells us
that the basis functions of $\mathcal{H}_{h_{0}}^{\nu }\left( \mathbb{C}/%
\mathbb{Z}\right) $ are 
\begin{equation}
V_{h_{0}}(e_{k,\nu })(x,\xi )=e^{2i\pi x(\nu +k-\xi )-\pi (\nu +k-\xi )^{2}}%
\text{.}  \label{basis}
\end{equation}%
Thus, using Proposition \ref{Proposition1} (4), the reproducing kernel is%
\begin{equation*}
K_{h_{0}}^{\nu }(z,w)=e^{i\pi (x\xi -x^{\prime }\xi ^{\prime })}\sum_{k\in 
\mathbb{Z}}e^{2i\pi (\nu +k)(x-x^{\prime })}e^{-\pi \lbrack (\nu +k-\xi
)^{2}+(\nu +k-\xi ^{\prime })^{2}]}
\end{equation*}%
alternatively, using (\ref{repkernel1}), 
\begin{equation*}
K_{h_{0}}^{\nu }(z,w)=\sum_{k\in \mathbb{Z}}e^{2\pi ik\nu }e^{-2\pi i\xi
^{\prime }k}K_{h_{0}}(z,w+k)\text{.}
\end{equation*}%
Inserting the expression (\ref{kernelGaussian}) for $K_{h_{0}}(z,w)$, this
leads to 
\begin{equation}
K_{h_{0}}^{\nu }(z,w)=e^{i\pi (x\xi -x^{\prime }\xi ^{\prime })}e^{-\frac{%
\pi }{2}(\left\vert z\right\vert ^{2}+\left\vert w\right\vert ^{2})}e^{\pi z%
\overline{w}}\sum_{k\in \mathbb{Z}}e^{-2\pi ik\nu }e^{-\pi \overline{w}k+\pi
kz-\frac{\pi }{2}k^{2}}\text{.}  \label{RK0}
\end{equation}
\end{example}

\begin{example}
\label{Example1}Since $\mathcal{F}(h_{r})=i^{r}h_{r}$, the basis functions
of $\mathcal{H}_{h_{r}}^{\nu }\left( \mathbb{C}/\mathbb{Z}\right) $ are 
\begin{equation*}
V_{h_{r}}(e_{k,\nu })(x,\xi )=e^{2i\pi x(\nu +k-\xi )}i^{r}h_{r}(\xi -\nu -k)%
\text{%
\'{}%
.}
\end{equation*}%
From\ Proposition \ref{Proposition1} (4),%
\begin{equation*}
K_{h_{r}}^{\nu }(z,w)=e^{i\pi (x\xi -x^{\prime }\xi ^{\prime })}\sum_{k\in 
\mathbb{Z}}(-1)^{r}e^{2i\pi (\nu +k)(x-x^{\prime })}h_{r}(\xi -\nu
-k)h_{r}(\xi ^{\prime }-\nu -k)\text{,}
\end{equation*}%
or, using (\ref{repkernel1}),%
\begin{equation*}
K_{h_{r}}^{\nu }(z,w)=\sum_{k\in \mathbb{Z}}e^{2\pi ik\nu }e^{-2\pi i\xi
^{\prime }k}K_{h_{r}}(z,w+k)\text{.}
\end{equation*}%
Inserting (\ref{kernelHermite}), 
\begin{equation}
K_{h_{r}}^{\nu }(z,w)=e^{i\pi (x\xi -x^{\prime }\xi ^{\prime })}e^{-\frac{%
\pi }{2}(\left\vert z\right\vert ^{2}+\left\vert w\right\vert ^{2})}e^{\pi 
\overline{w}z}\sum_{k\in \mathbb{Z}}e^{2\pi ik\nu }e^{\pi (\overline{w}%
+z)k}e^{-\frac{\pi }{2}k^{2}}L_{r}(\pi \left\vert z-w-k\right\vert ^{2})%
\text{.}  \label{RKr}
\end{equation}
\end{example}

\section{Fock-type spaces on $\mathbb{C}/\mathbb{Z}$}

In this section we will consider subspaces of $L^{2}(\mathbb{C})$, $L^{2}(%
\mathbb{C},e^{-\pi |z|^{2}})$, $L^{2}(\mathbb{C}/\mathbb{Z})$ and $L^{2}(%
\mathbb{C}/\mathbb{Z},e^{-\pi |z|^{2}})$ and define isometric isomorphisms%
\begin{equation}
M:L^{2}(\mathbb{C}/\mathbb{Z})\rightarrow L^{2}(\mathbb{C}/\mathbb{Z}%
,e^{-\pi |z|^{2}})\text{, and\ \ }M:L^{2}(\mathbb{C})\rightarrow L^{2}(%
\mathbb{C},e^{-\pi |z|^{2}})\text{,}  \label{M}
\end{equation}%
where $(MF)(z)=M(z)F(z)$, $M(z)=e^{-i\pi x\xi +\frac{\pi }{2}\left\vert
z\right\vert ^{2}}$.

\subsection{The Bargmann-Fock space on $\mathbb{C}/\mathbb{Z}$}

Take $g(t)=h_{0}=2^{\frac{1}{4}}e^{-\pi t^{2}}$ so that the classical
Bargmann transform 
\begin{equation}
\mathcal{B}f(z)=e^{-\frac{\pi }{2}z^{2}}\int_{\mathbb{R}}f(t)e^{2\pi tz-\pi
t^{2}}dt\text{,}  \label{BarDef0}
\end{equation}%
can be written as 
\begin{equation}
\mathcal{B}f(z)=e^{-i\pi x\xi +\frac{\pi }{2}\left\vert z\right\vert
^{2}}V_{h_{0}}f(x,-\xi )=e^{-i\pi x\xi +\frac{\pi }{2}\left\vert
z\right\vert ^{2}}V_{h_{0}}f(\overline{z})\text{,}\ z=x+i\xi \text{.}
\label{BargDef}
\end{equation}%
Let $\mathcal{F}\left( \mathbb{C}\right) $ be the Fock space of entire
functions $F(z)$ such that%
\begin{equation*}
\int_{\mathbb{C}}|F(z)|^{2}e^{-\pi \left\vert z\right\vert ^{2}}dz<\infty 
\text{.}
\end{equation*}%
Then $\mathcal{B}:{L}^{2}(\mathbb{R})\rightarrow \mathcal{F}(\mathbb{C})$ is
an unitary isomorphism \cite[Theorem 3.4.3]{Charly}, so that $\mathcal{F}%
\left( \mathbb{C}\right) $ can alternately be defined as 
\begin{equation*}
\mathcal{F}\left( \mathbb{C}\right) =\mathcal{B}\left( {L}^{2}(\mathbb{R}%
)\right) =MV_{h_{0}}\left( {L}^{2}(\mathbb{R})\right) \text{.}
\end{equation*}%
Since $h_{0}\in {S}_{0}(\mathbb{R})$, then $V_{h_{0}}f\in L_{1}(\mathbb{C}%
)\subset L_{2}(\mathbb{C})$ the kernel of the Bargmann transform $t\mapsto
A(z,t):=e^{2\pi tz-\pi t^{2}-\frac{\pi }{2}z^{2}}$ is rapidly decreasing, it
belongs to $\mathcal{S(\mathbb{R})}$. This allows to interpret the brackets 
\begin{equation*}
\mathcal{B}f(z)=\left\langle A(z,.),f\right\rangle \text{, }f\in \mathcal{S}%
^{\prime }(\mathbb{R})
\end{equation*}%
as a duality pairing. Since this definition also holds for pairs of test
function and distribution spaces, like the Schwartz space and tempered
distributions $\mathcal{S(\mathbb{R})}$, $\mathcal{S}^{\prime }(\mathbb{R})$~%
\cite{Charly} and the Feichtinger algebra ${S}_{0}(\mathbb{R})$ \cite{Fei}
and its dual ${S}_{0}^{\prime }(\mathbb{R})$, we see that is well defined
for $f\in {L}_{\nu }^{2}(0,1)$. Combining (\ref{BargDef}) with (\ref%
{almostperiodic}), gives%
\begin{equation}
\mathcal{B}f(z+k)=e^{2\pi ik\nu }e^{\pi kz+\frac{\pi }{2}k^{2}}\mathcal{B}%
f(z)\text{.}  \label{invarianceFock}
\end{equation}%
Now consider $\mathcal{F}_{2,\nu }\left( \mathbb{C}/\mathbb{Z}\right) $, the
Fock space of entire functions $F(z)$ such that%
\begin{equation}
\int_{\lbrack 0,1]\times \mathbb{R}}|F(z)|^{2}e^{-\pi \left\vert
z\right\vert ^{2}}dz<\infty  \label{intL2}
\end{equation}%
and satisfying the functional equation 
\begin{equation}
F(z+k)=e^{2\pi ik\nu }e^{\frac{\pi }{2}k^{2}+\pi zk}F(z)  \label{funct-equa}
\end{equation}%
for all $z\in \mathbb{C}$ and $k\in \mathbb{Z}$. Then (\ref{intL2}) is well
defined and independent of the choice of the domain representing the
fundamental region of $\mathbb{C}/\mathbb{Z}$. In particular, it follows
from (\ref{funct-equa}) that $\mathcal{B}f(z)$ satisfy (\ref{funct-equa}).
Alternatively, writing $F(z)=e^{\frac{\pi }{2}z^{2}+2i\pi \nu z}g(z)$, $%
\mathcal{F}_{2,\nu }\left( \mathbb{C}/\mathbb{Z}\right) $, is isomorphic to
the space of functions $g(x,y)$ periodic in the $x$-direction and such that 
\cite[Remark 2.2]{GhanmiIntissar2013}: 
\begin{equation}
\Vert f\Vert _{\mathcal{F}_{2,\nu }(\mathbb{C}/\mathbb{Z})}^{2}=\int_{[0,1]}%
\int_{\mathbb{R}}|g(z)|^{2}e^{-2\pi y^{2}-4\pi \nu y}dxdy\text{.}
\label{normaltern}
\end{equation}%
\ The space $\mathcal{F}_{2,\nu }\left( \mathbb{C}/\mathbb{Z}\right) $ has
an orthogonal basis given by \cite[Theorem 2.3]{GhanmiIntissar2013}: 
\begin{equation}
\varphi _{k,\nu }(z):=e^{\frac{\pi }{2}z^{2}+2i\pi z(\nu +k)-\pi (\nu
+k)^{2}}\text{,}\quad k\in \mathbb{Z}\text{,}  \label{basisF}
\end{equation}%
which is the image of the orthonormal basis $\{e_{k,\nu }(t)=e^{2\pi it(\nu
+k)}\}_{k\in \mathbb{Z}}$ of ${L}_{\nu }^{2}(0,1)$ under $\mathcal{B}$.
Indeed, from (\ref{basis}) and (\ref{BargDef}) we obtain 
\begin{equation*}
\mathcal{B}\left( e_{k,\nu }\right) (z)=e^{-i\pi x\xi +\frac{\pi }{2}%
\left\vert z\right\vert ^{2}}V_{h_{0}}(e_{k,\nu })(x,-\xi )=\varphi _{k,\nu
}(z)\text{.}
\end{equation*}%
Therefore, the classical Bargmann transform is an unitary isomorphism from ${%
L}_{\nu }^{2}(0,1)$ onto the space $\mathcal{F}_{2,\nu }\left( \mathbb{C}/%
\mathbb{Z}\right) $ of holomorphic functions on $\mathbb{C}$ satisfying (\ref%
{funct-equa}). The reproducing kernel of $\mathcal{F}_{2,\nu }\left( \mathbb{%
C}/\mathbb{Z}\right) $ can be written as the normalized bilinear sum of the
basis functions and we are lead to:{%
\begin{equation*}
K_{\mathcal{F}_{2,\nu }\left( \mathbb{C}/\mathbb{Z}\right) }(z,w)=\left(
2\right) ^{1/2}e^{\frac{\pi }{2}(z^{2}+\overline{w}^{2})}\sum_{k\in \mathbb{Z%
}}e^{-2\pi (k+\nu )^{2}+2i\pi (k+\nu )(z-\overline{w})}
\end{equation*}%
which can also be written as}%
\begin{equation*}
K_{\mathcal{F}_{2,\nu }\left( \mathbb{C}/\mathbb{Z}\right) }(z,w)=\left(
2\right) ^{1/2}e^{\frac{\pi }{2}(z^{2}+\overline{w}^{2})}\theta _{\nu
,0}\left( z-\overline{w}\left\vert i\right. \right) \text{,}
\end{equation*}%
{where }$\theta _{\nu ,0}$ is the classical Jacobi theta function. An
alternative expression for the reproducing kernel is the following Poincar%
\'{e} periodization of the kernel of the Fock space of entire functions: 
\begin{equation*}
K_{\mathcal{F}_{2,\nu }\left( \mathbb{C}/\mathbb{Z}\right) }(z,w)=e^{\pi z%
\overline{w}}\sum_{k\in \mathbb{Z}}e^{-2\pi ik\nu }e^{-\pi \overline{w}k+\pi
kz-\frac{\pi }{2}k^{2}}\text{.}
\end{equation*}%
Using (\ref{RK0}), $K_{\mathcal{F}_{2,\nu }\left( \mathbb{C}/\mathbb{Z}%
\right) }(z,w)$ can be related to $K_{h_{0}}^{\nu }(z,w)$ as follows 
\begin{equation*}
K_{h_{0}}^{\nu }(z,w)=\sum_{k\in \mathbb{Z}}e^{2\pi ik\nu }e^{-2\pi i\xi
^{\prime }k}K_{h_{0}}(z,w+k)=e^{i\pi (x\xi -x^{\prime }\xi ^{\prime })}e^{-%
\frac{\pi }{2}(\left\vert z\right\vert ^{2}+\left\vert w\right\vert ^{2})}K_{%
\mathcal{F}_{2,\nu }\left( \mathbb{C}/\mathbb{Z}\right) }(z,w)\text{.}
\end{equation*}

\subsection{True polyanalytic Fock spaces on $\mathbb{C}/\mathbb{Z}$}

True polyanalytic Bargmann-type transforms $\mathcal{B}^{\left( r\right) }$
can be defined \cite{Abr2010} as%
\begin{equation*}
\mathcal{B}^{\left( r\right) }f(z)=\left( \frac{\pi ^{r}}{r!}\right) ^{\frac{%
1}{2}}e^{\pi \left\vert z\right\vert ^{2}}\left( \partial _{z}\right) ^{r}%
\left[ e^{-\pi \left\vert z\right\vert ^{2}}\mathcal{B}f(z)\right] \text{.}
\end{equation*}%
Setting $g=h_{r}$ in the definition of the STFT, we obtain the following
extension of the relation (\ref{BargDef}): 
\begin{equation}
\mathcal{B}^{\left( r\right) }f(z)=e^{-i\pi x\xi +\frac{\pi }{2}\left\vert
z\right\vert ^{2}}V_{h_{r}}f(x,-\xi )\text{,}\ z=x+i\xi \text{.}
\label{trupolydef}
\end{equation}%
The transform $\mathcal{B}^{\left( r\right) }$ defines a sequence of spaces $%
\mathcal{F}_{2}^{\left( r\right) }\left( \mathbb{C}\right) \subset L^{2}(%
\mathbb{C},e^{-\pi |z|^{2}})$: 
\begin{equation*}
\mathcal{F}_{2}^{\left( r\right) }\left( \mathbb{C}\right) =\mathcal{B}%
^{\left( r\right) }\left[ {L}^{2}(\mathbb{R})\right] =MV_{h_{r}}\left[ {L}%
^{2}(\mathbb{R})\right]
\end{equation*}%
with reproducing kernel%
\begin{equation}
K_{\mathcal{F}_{2}^{\left( r\right) }\left( \mathbb{C}\right) }(z,w)=e^{\pi 
\overline{w}z}L_{r}(\pi \left\vert z-w\right\vert ^{2})\text{.}
\label{repTrue}
\end{equation}%
These spaces are of particular importance since they are the eigenspaces of
the Landau Hamiltonian with a constant magnetic field, 
\begin{equation}
L_{z,\overline{z}}:=-\partial _{z}\partial _{\overline{z}}+\pi \overline{z}%
\partial _{\overline{z}}  \label{Landau-hami}
\end{equation}%
associated with the eigenvalue $\pi r$. More precisely \cite{Zouhair,Mouayn},%
\begin{equation*}
\mathcal{F}_{2}^{\left( r\right) }\left( \mathbb{C}\right) =\left\{ F\in
D(L_{z})\subset L^{2}(\mathbb{C},e^{-\pi |z|^{2}}):L_{z}F(z)=\pi
rF(z)\right\} \text{.}
\end{equation*}%
The Landau Hamiltonian is a relevant model for the physics of the Integer
Quantum-Hall effect, with a rich mathematical structure, which, in recent
years, has been explored from a variety of perspectives. See \cite{AbrSpeck}
for double orthogonality and $L_{1}$-minimization problems, \cite%
{HendHaimi,abgrro17} for associated point processes, \cite{AoP} for discrete
coherent states, \cite{Kord} for a Berezin transform approach, \cite{Onofri}
for toric versions, \cite{Ruz} for weak formulations and \cite{IZ2,Ziyat}
for the quasi-tori setting, where the Hamiltonian is the Bochner Laplacian
on sections of the line bundle.

From (\ref{trupolydef}) and (\ref{almostperiodic}), restricting to $f\in {L}%
_{\nu }^{2}(0,1)$, gives%
\begin{equation*}
\mathcal{B}^{\left( r\right) }f(z+k)=e^{2\pi ik\nu }e^{\pi kz+\frac{\pi }{2}%
\left\vert k\right\vert ^{2}}\mathcal{B}^{\left( r\right) }f(z)\text{.}
\end{equation*}%
As a result, the transforms $\mathcal{B}^{\left( r\right) }$ define a
sequence of spaces $\mathcal{F}_{\nu ,r}^{\left( r\right) }\left( \mathbb{C}/%
\mathbb{Z}\right) $, consisting of functions $F$ of the form%
\begin{equation}
F(z)=\left( \frac{\pi ^{r}}{r!}\right) ^{\frac{1}{2}}e^{\pi \left\vert
z\right\vert ^{2}}\left( \partial _{z}\right) ^{r}\left[ e^{-\pi \left\vert
z\right\vert ^{2}}F_{0}(z)\right] \text{,}  \label{TruePolyRepresentation}
\end{equation}%
for some $F_{0}\in \mathcal{F}_{2,\nu }\left( \mathbb{C}/\mathbb{Z}\right) $%
\ (the analytic Fock space of the previous section) and satisfying%
\begin{equation*}
F(z+k)=e^{2\pi ik\nu }e^{\pi kz+\frac{\pi }{2}\left\vert k\right\vert
^{2}}F(z)\text{.}
\end{equation*}%
The spaces $\mathcal{F}_{2,\nu }^{\left( r\right) }\left( \mathbb{C}/\mathbb{%
Z}\right) $ are, up to the multiplier $M$ defined by (\ref{M}) the STFT with
Hermite functions images of ${L}_{\nu }^{2}(0,1)$: 
\begin{equation*}
\mathcal{F}_{2,\nu }^{\left( r\right) }\left( \mathbb{C}/\mathbb{Z}\right) =%
\mathcal{B}^{\left( r\right) }\left[ {L}_{\nu }^{2}(0,1)\right] =M\mathcal{H}%
_{h_{r}}^{\nu }\left( \mathbb{C}/\mathbb{Z}\right) \text{.}
\end{equation*}%
From (\ref{Moyal}), given $f_{1},f_{2}\in {L}_{\nu }^{2}(0,1)$ and $%
r_{1}\neq r_{2}$, 
\begin{equation}
\left\langle V_{h_{r_{1}}}f_{1},V_{h_{r_{2}}}f_{2}\right\rangle
_{L^{2}((0,1)\times \mathbb{R})}=\left\langle f_{1},f_{2}\right\rangle
_{L^{2}(0,1)}\left\langle h_{r_{1}},h_{r_{2}}\right\rangle _{L^{2}(\mathbb{R}%
)}=0\text{.}  \label{ortTruepoly}
\end{equation}%
Therefore, the spaces $\mathcal{H}_{h_{r}}^{\nu }\left( \mathbb{C}/\mathbb{Z}%
\right) $ are orthogonal in $L^{2}(\mathbb{C}/\mathbb{Z})=L^{2}((0,1)\times 
\mathbb{R})$ and, consequently, $\mathcal{F}_{2,\nu }^{\left( r\right)
}\left( \mathbb{C}/\mathbb{Z}\right) $ are orthogonal in $L^{2}(\mathbb{C}/%
\mathbb{Z},e^{-\pi |z|^{2}})$ for different values of $r$. From (\ref{RKr}),
the reproducing kernel of $\mathcal{H}_{h_{r}}^{\nu }\left( \mathbb{C}/%
\mathbb{Z}\right) $ is 
\begin{equation*}
K_{h_{r}}^{\nu }(\overline{z},\overline{w})=e^{i\pi (x\xi -x^{\prime }\xi
^{\prime })}e^{-\frac{\pi }{2}(\left\vert z\right\vert ^{2}+\left\vert
w\right\vert ^{2})}e^{\pi \overline{w}z}\sum_{k\in \mathbb{Z}}e^{2\pi ik\nu
}e^{\pi (\overline{w}+z)k}e^{-\frac{\pi }{2}k^{2}}L_{r}(\pi \left\vert
z-w-k\right\vert ^{2})\text{.}
\end{equation*}%
As a consequence, the reproducing kernel of $\mathcal{F}_{\nu }^{(r)}\left( 
\mathbb{C}/\mathbb{Z}\right) $ is given as 
\begin{equation*}
K_{\mathcal{F}_{\nu }^{(r)}\left( \mathbb{C}/\mathbb{Z}\right) }(z,w)=M(z)M(%
\overline{w})K_{h_{r}}^{\nu }(\overline{z},\overline{w})=e^{\pi \overline{w}%
z}\sum_{k\in \mathbb{Z}}e^{2\pi ik\nu }e^{\pi (\overline{w}+z)k}e^{-\frac{%
\pi }{2}k^{2}}L_{r}(\pi \left\vert z-w-k\right\vert ^{2})\text{.}
\end{equation*}%
Denote the basis functions of ${L}_{\nu }^{2}(0,1)$ by $f_{k}(t)=e^{2\pi
it(\nu +k)}$.\ Their transforms by $\mathcal{B}^{\left( r\right) }$ follow
from (\ref{basis_TF}), and%
\begin{equation*}
\varphi _{k}^{(r)}(z)=\mathcal{B}^{\left( r\right) }\left( f_{k}\right)
=e^{-i\pi x\xi +\frac{\pi }{2}\left\vert z\right\vert
^{2}}V_{h_{r}}(f_{k})(x,-\xi )=e^{-i\pi x\xi +\frac{\pi }{2}\left\vert
z\right\vert ^{2}}e^{2i\pi x(\nu +n+\xi )}i^{r}h_{r}(-\xi -\nu -n)\text{,}
\end{equation*}%
provide an orthogonal basis of $\mathcal{F}_{2,\nu }^{\left( r\right)
}\left( \mathbb{C}/\mathbb{Z}\right) $. By definition, they are complete in $%
\mathcal{F}_{2,\nu }^{\left( r\right) }\left( \mathbb{C}/\mathbb{Z}\right) $%
. Their orthogonality in $r$ and $k$ follows from (\ref{Moyal}). For $%
r_{1}\neq r_{2}$ or $k\neq m$, it follows from (\ref{ortTruepoly}) that 
\begin{equation}
\left\langle \varphi _{k}^{(r_{1})}(z),\varphi
_{m}^{(r_{2})}(z)\right\rangle _{L^{2}((0,1)\times \mathbb{R},e^{-\pi
|z|^{2}})}=\left\langle V_{h_{r_{1}}}f_{k},V_{h_{r_{2}}}f_{m}\right\rangle
_{L^{2}((0,1)\times \mathbb{R})}=0\text{.}  \label{doubleindexort}
\end{equation}

\begin{remark}
These functions also show up in (23) of Theorem 4.1 in \cite{IZ1}.
\end{remark}

\subsection{The complex periodization operator}

Despite this paper dealing essentially with the Hilbert spaces, we will need
to consider some Banach Fock spaces in hogher Landau levels of $\mathbb{C}/%
\mathbb{Z}$ for our proofs. Banach Fock spaces in higher Landau levels in $%
\mathbb{C}$ have been considered in \cite{AbrGr2012}, in connection with
Banach frames with Hermite functions. We are currently working on a similar
treatment in $\mathbb{C}/\mathbb{Z}$ \cite{AZLBanach}, but in this paper we
will only introduce and prove the essential properties required for the
proof of our main results.

We start defining $p$-versions of the space $\mathcal{F}_{2,\nu }(\mathbb{C}/%
\mathbb{Z})$, for\ $p\in \lbrack 1,\infty \lbrack $ as the spaces of
analytic functions on $\mathbb{C}\ $with the norm%
\begin{equation}
\left\Vert F\right\Vert _{\mathcal{F}_{p,\nu }(\mathbb{C}/\mathbb{Z}%
)}=\int_{(0,1)\times \mathbb{R}}|F(z)|^{p}e^{-\pi \frac{p}{2}|z|^{2}}dz\text{%
.}  \label{integral}
\end{equation}%
satisfying 
\begin{equation}
F(z+k)=e^{2\pi ik\nu }e^{\frac{\pi }{2}|k|^{2}+\pi zk}F(z)\text{.}
\label{functv}
\end{equation}%
This implies%
\begin{equation}
|F(z-k)|^{p}e^{-\pi \frac{p}{2}|z-k|^{2}}=|F(z)|^{p}e^{-\pi \frac{p}{2}%
|z|^{2}}\text{.}  \label{functp}
\end{equation}%
The $\mathbb{Z}$-periodicity\ of $|F(z)|^{p}e^{-\pi \frac{p}{2}|z|^{2}}$
assures that the integral (\ref{integral}) makes sense and is independent of
the fundamental domain representing $\mathbb{C}/\mathbb{Z}$. For a Weyl
translation of $F\in \mathcal{F}_{p,\nu }(\mathbb{C}/\mathbb{Z})$ to be
well-defined, the imaginary part of the translation parameter $w$ must be
integer: $w\in \mathbb{R+}i\mathbb{Z}$. Thus,%
\begin{equation*}
\tau (w)F(z)=e^{\pi z\overline{w}-\frac{\pi }{2}\left\vert w\right\vert
^{2}}F(z-w)\text{, }\Im w\in \mathbb{Z}\text{.}
\end{equation*}%
Then, $\tau (w)F(z)$ satisfies (\ref{funct-equa}):%
\begin{equation*}
\tau (w)F(z-k)=e^{2\pi ik\nu }e^{\pi kz+\frac{\pi }{2}\left\vert
k\right\vert ^{2}}\tau (w)F(z)\text{,}
\end{equation*}%
thus%
\begin{equation*}
\left\vert \tau (w)F(z-k)\right\vert ^{p}e^{-\pi \frac{p}{2}%
|z-k|^{2}}=\left\vert F(z-w-k)\right\vert ^{p}e^{-\pi \frac{p}{2}%
|z-w-k|^{2}}=\left\vert \tau (w)F(z)\right\vert ^{p}e^{-\pi \frac{p}{2}%
|z|^{2}}\text{.}
\end{equation*}%
and $\tau (w)F(z)$ is entire in $z$ and $\mathbb{Z}$-periodic. Therefore, $%
\tau (w)F(z)\in \mathcal{F}_{p,\nu }(\mathbb{C}/\mathbb{Z})$ and the spaces $%
\mathcal{F}_{p,\nu }(\mathbb{C}/\mathbb{Z})$ are invariant under the Weyl
translation. The following operator is a `phase space' version of the
operator (\ref{Period}), in the sense that it can be obtained from the
latter by the action of the Bargmann transform. One can generate elements of 
$\mathcal{F}_{1,\nu }(\mathbb{C}/\mathbb{Z})$ by periodizing functions in $%
\mathcal{F}_{1}(\mathbb{C})$. On the other hand, the operator can also be
used to show the existence of a function in $\mathcal{F}_{1}(\mathbb{C})$,
whose periodization is a given function belonging to $\mathcal{F}_{1,\nu }(%
\mathbb{C}/\mathbb{Z})$. This route will be a key step in the proof of
Theorem \ref{FrameHermite}.

\begin{lemma}
\label{ComplexPeriodization}Let $\Gamma ^{\prime }=i\frac{1}{\beta }\mathbb{Z%
}$. Then the periodization operator 
\begin{equation}
\mathcal{P}_{\Gamma ^{\prime },l}(F)(z):=\sum_{\gamma \in \Gamma ^{\prime
}}\chi _{l}(\gamma )e^{-\pi z\overline{\gamma }-\frac{\pi }{2}|\gamma
|^{2}}F(z+\gamma )  \label{per}
\end{equation}%
is bounded and onto from $\mathcal{F}_{1}(\mathbb{C})$ into $\mathcal{F}%
_{1,l}(\mathbb{C}/\Gamma ^{\prime })$.
\end{lemma}

\begin{proof}
This is proved in \cite{E-G-I-Z}, but since the paper has not yet been
published, we give a direct proof. Namely, it is easy to see that this
operator is bounded with constant bound less than one. Furthermore, by
considering the change of coordinate $z=\frac{i}{\beta }w$, we can consider $%
\mathcal{P}_{\Gamma }:\mathcal{F}_{1}(\mathbb{C})\rightarrow \mathcal{F}_{1}(%
\mathbb{C}/\Gamma )$; with $\Gamma =\mathbb{Z}$ and $\nu =0$. We consider
the functions 
\begin{equation*}
\varphi _{k}(x)=e^{2i\pi kx+\pi k^{2}}\chi _{\lbrack 0,1]},\quad k\in 
\mathbb{Z}\text{.}
\end{equation*}%
Then the family of functions $f_{k}(z)=\mathcal{B}\varphi _{k}(z)$ belongs
to the space $\mathcal{F}_{1}(\mathbb{C})$. By simple computation we obtain 
\begin{equation}
\mathcal{P}_{\Gamma }f_{k}(z)=e^{\frac{\pi }{2}z^{2}+2i\pi kz}\text{.}
\end{equation}%
Thus, it remains to show that the family of functions $e^{\frac{\pi }{2}%
z^{2}+2i\pi kz}$ is dense in $\mathcal{F}_{1}(\mathbb{C}/\mathbb{Z})$. Since 
$\mathcal{F}_{1}(\mathbb{C}/\mathbb{Z})\subset \mathcal{F}_{2}(\mathbb{C}/%
\mathbb{Z})$ then by using the reproducing kernel representation, we get for 
$F\in \mathcal{F}_{1}(\mathbb{C}/\mathbb{Z})$ 
\begin{equation*}
F(z)=\left\langle F,\overline{K_{\mathcal{F}_{2}\left( \mathbb{C}/\mathbb{Z}%
\right) }(z,w)}\right\rangle _{\mathcal{F}_{2}(\mathbb{C}/\mathbb{Z}%
)}=\left( 2\right) ^{1/2}\sum_{n\in \mathbb{Z}}a_{n}e^{\frac{\pi }{2}%
z^{2}+2i\pi nz}\text{,}
\end{equation*}%
with $a_{n}=e^{-2\pi n^{2}}\left\langle F(w),e^{\frac{\pi }{2}w^{2}+2i\pi
nw}\right\rangle _{\mathcal{F}_{2}(\mathbb{C}/\mathbb{Z})}$. Note that 
\begin{equation*}
\left\vert a_{n}\right\vert \leq e^{-\pi n^{2}}\Vert F\Vert _{\mathcal{F}%
_{1}(\mathbb{C}/\Gamma )}\text{.}
\end{equation*}%
Thus, setting 
\begin{equation*}
P_{N}(z)=\left( 2\right) ^{1/2}\sum_{|n|\leq N}a_{n}e^{\frac{\pi }{2}%
z^{2}+2i\pi kz}\text{,}
\end{equation*}%
we obtain 
\begin{equation}
\Vert F-P_{N}\Vert _{\mathcal{F}_{1}(\mathbb{C}/\mathbb{Z})}\leq \Vert
F\Vert _{\mathcal{F}_{1}(\mathbb{C}/\mathbb{Z})}\sum_{|n|>N}e^{-\pi n^{2}}%
\text{,}
\end{equation}%
and the term in right hand side of the above inequality go to zero when $%
N\rightarrow +\infty $. Thus, $\mathcal{P}_{\Gamma }:\mathcal{F}_{1}(\mathbb{%
C})\rightarrow \mathcal{F}_{1}(\mathbb{C}/\Gamma )$ has dense range. Now,
given $F\in \mathcal{F}_{1}(\mathbb{C}/\mathbb{Z})$, then $e^{-\frac{\pi }{2}%
z^{2}}F(z)$\ is periodic and analytic and therefore can be uniformly
approximated by its partial Fourier series $P_{N}(z)$. We conclude that $%
\mathcal{P}_{\Gamma }:\mathcal{F}_{1}(\mathbb{C})\rightarrow \mathcal{F}_{1}(%
\mathbb{C}/\Gamma )$ is onto.
\end{proof}

\section{The structure of Gabor frames on ${L}_{\protect\nu }^{2}(0,1)$}

\textbf{Structure of Gabor frames in}\textsl{\ }$L^{2}\left( \mathbb{R}%
\right) $\textbf{. }Given a regular lattice $\Lambda _{\alpha ,\beta
}=\alpha \mathbb{Z}\times \beta \mathbb{Z}$, the dual lattice is $\Lambda
_{\alpha ,\beta }^{0}=\frac{1}{\alpha }\mathbb{Z}\times \frac{1}{\beta }%
\mathbb{Z}$. Gabor frames indexed by lattices have a beautiful structure,
leading to useful properties that lie at the heart of the rich mathematical
theory of time-frequency analysis. A chief example is the \emph{Ron-Shen
duality principle} \cite[Theorem 7.4.3]{Charly}, stating that if $\mathcal{G}%
\left( g,\Lambda \right) $ is a Gabor frame for $L^{2}\left( \mathbb{R}%
\right) $ if and only if $\mathcal{G}\left( g,\Lambda ^{0}\right) $ is a
Gabor Riesz sequence for $L^{2}\left( \mathbb{R}\right) $, where $\Lambda
^{0}$ is the dual lattice. This is a consequence of the \emph{Wexler-Raz
biorthogonality relations}, which state that $\mathcal{G}\left( g,\Lambda
_{\alpha ,\beta }\right) $ is a Gabor frame for $L^{2}\left( \mathbb{R}%
\right) $ \emph{if and only if there exists }$\gamma \in L^{2}\left( \mathbb{%
R}\right) $ such that\emph{\ }%
\begin{equation}
\beta ^{-1}\left\langle \gamma ,M_{\frac{n}{\alpha }}T_{\frac{l}{\beta }%
}g\right\rangle =\delta _{n,0}\delta _{l,0}\text{ \ for all }n,l\in \mathbb{Z%
}\text{.}  \label{WR}
\end{equation}

\textbf{Structure of Gabor frames in}\textsl{\ }$L_{\nu }^{2}(0,1)$\textbf{.}
In this section we will define Gabor frames for ${L}_{\nu }^{2}(0,1)$ and
prove analogues of some of the structural properties of Gabor frames for $%
L^{2}\left( \mathbb{R}\right) $, most notably of \emph{the sufficient part
of the Wexler-Raz biorthogonality relations} (Corollary \ref{WexlerRaz}
below). This will be used to provide, for Hermite and totally positive
functions, sufficient conditions on $\beta $, assuring that $\mathcal{G}%
_{\nu }\left( g,Z_{\beta }\right) $ is a Gabor frame for ${L}_{\nu
}^{2}(0,1) $.

\subsection{Definition}

We will investigate Gabor systems in ${L}_{\nu }^{2}(0,1)$ of the form 
\begin{equation}
\mathcal{G}_{\nu }\left( g,Z\right) :=\{\Sigma _{\nu }\left( \pi (z)g\right)
,\;z\in Z\}\text{,}  \label{Frame_v}
\end{equation}%
for a general $g\in \mathcal{S}_{0}(\mathbb{R})$, with $\left\Vert
g\right\Vert _{L^{2}(\mathbb{R})}=1$, where $Z$ is the regular lattice $%
Z_{\beta }:=\beta \mathbb{Z\subset }\Lambda (\mathbb{Z}):=[0,1)\times 
\mathbb{R}$, such that\ $\pi (\beta n)g=M_{\beta n}g$, to establish Gabor
frame properties in ${L}_{\nu }^{2}(0,1)$.

In this subsection we will justify that (\ref{Frame_v}) is the right
definition for a Gabor system in ${L}_{\nu }^{2}(0,1)$. For a given sequence 
$Z$ of distinct numbers in the fundamental domain representing $\mathbb{C}/%
\mathbb{Z}$, we start by defining a sampling sequence for the space $%
\mathcal{H}_{g}^{\nu }\left( \mathbb{C}/\mathbb{Z}\right) $, whenever there
exist constants $A,B>0$ such that, for all $f\in {L}_{\nu }^{2}(0,1)$\ the
sampling inequality holds:%
\begin{equation}
A\left\Vert V_{g}f\right\Vert _{\mathcal{H}_{g}^{\nu }\left( \mathbb{C}/%
\mathbb{Z}\right) }^{2}\leq \sum_{z\in Z}\left\vert V_{g}f(z)\right\vert
^{2}\leq B\left\Vert V_{g}f\right\Vert _{\mathcal{H}_{g}^{\nu }\left( 
\mathbb{C}/\mathbb{Z}\right) }^{2}\text{.}  \label{samp}
\end{equation}%
We will show that this sampling inequality if equivalent to the Gabor system
in ${L}_{\nu }^{2}(0,1)$ of the form (\ref{Frame_v}). Using the Moyal
formulas (\ref{Moyal}), since $\left\Vert g\right\Vert _{L^{2}(\mathbb{R}%
)}=1 $,\ this is equivalent to saying that\ there exist constants $A,B>0$
such that, for all $f\in {L}_{\nu }^{2}(0,1)$,%
\begin{equation}
A\left\Vert f\right\Vert _{{L}_{\nu }^{2}(0,1)}^{2}\leq \sum_{z\in
Z}\left\vert \left\langle f,\pi (z)g\right\rangle _{L^{2}\left( \mathbb{R}%
\right) }\right\vert ^{2}\leq B\left\Vert f\right\Vert _{{L}_{\nu
}^{2}(0,1)}^{2}\text{.}  \label{frame}
\end{equation}

Now, for all $f\in {L}_{\nu }^{2}(0,1)$, the quasi-periodicity condition (%
\ref{periodic}) can be written as $e^{-2\pi ik\nu }f=T_{k}f$. The following
auxiliary result will be used often in this section.

\begin{lemma}
If $g\in \mathcal{S}_{0}(\mathbb{R})$ and $f\in {L}_{\nu }^{2}(0,1)$, then 
\begin{equation}
\left\langle f,\Sigma _{\nu }\left( \pi (z)g\right) \right\rangle _{{L}_{\nu
}^{2}(0,1)}=\left\langle f,\pi (z)g\right\rangle _{L^{2}\left( \mathbb{R}%
\right) }\text{.}  \label{ident}
\end{equation}
\end{lemma}

\begin{proof}
Thus, if $g\in \mathcal{S}_{0}(\mathbb{R})$, by Proposition \ref%
{Proposition0}, $\Sigma _{\nu }\left( \pi (z)g\right) \in {L}_{\nu
}^{2}(0,1) $ and we have 
\begin{eqnarray*}
\left\langle f,\Sigma _{\nu }\left( \pi (z)g\right) \right\rangle _{{L}_{\nu
}^{2}(0,1)} &=&\sum_{k\in \mathbb{Z}}\left\langle f,e^{2\pi ik\nu }T_{k}\pi
(z)g\right\rangle _{{L}_{\nu }^{2}(0,1)} \\
&=&\sum_{k\in \mathbb{Z}}\left\langle e^{-2\pi ik\nu }f,T_{k}\pi
(z)g\right\rangle _{{L}_{\nu }^{2}(0,1)} \\
&=&\sum_{k\in \mathbb{Z}}\left\langle T_{k}f,T_{k}\pi (z)g\right\rangle _{{L}%
_{\nu }^{2}(0,1)} \\
&=&\sum_{k\in \mathbb{Z}}\left\langle f,\left( \pi (z)g\right) \right\rangle
_{{L}^{2}(k,k+1)} \\
&=&\left\langle f,\left( \pi (z)g\right) \right\rangle _{L^{2}\left( \mathbb{%
R}\right) }\text{.}
\end{eqnarray*}
\end{proof}

Using (\ref{ident}), one concludes that the inequalities (\ref{frame}) can
be written as 
\begin{equation}
A\left\Vert f\right\Vert _{{L}_{\nu }^{2}(0,1)}^{2}\leq \sum_{z\in
Z}\left\vert \left\langle f,\Sigma _{\nu }\left( \pi (z)g\right)
\right\rangle _{{L}_{\nu }^{2}(0,1)}\right\vert ^{2}\leq B\left\Vert
f\right\Vert _{{L}_{\nu }^{2}(0,1)}^{2}\text{.}  \label{altFrame}
\end{equation}%
This justifies the following.

\begin{definition}
The Gabor system $\mathcal{G}_{\nu }\left( g,Z\right) :=\{\Sigma _{\nu
}\left( \pi (z)g\right) ,\;z\in Z\}$ is a \emph{Gabor frame} or \emph{%
Weyl-Heisenberg frame }in ${L}_{\nu }^{2}(0,1)$,\emph{\ }whenever there
exist constants $A,B>0$ such that, for all $f\in {L}_{\nu }^{2}(0,1)$, (\ref%
{altFrame}) holds.
\end{definition}

\subsection{The frame operator on ${L}_{\protect\nu }^{2}(0,1)$}

Let us consider $Z$ as the regular lattice $Z_{\beta }:=i\beta \mathbb{Z}%
\subset \lbrack 0,1)\times \mathbb{R}$. Equivalently, $Z_{\beta }$ consists
of the pairs $(0,\beta n)\subset \lbrack 0,1)\times \mathbb{R}$, $n\in 
\mathbb{Z}$. Then $\pi (z)g=M_{\beta n}g$ and%
\begin{equation*}
\mathcal{G}_{\nu }\left( g,Z\right) :=\{\Sigma _{\nu }\left( M_{\beta
n}g\right) ,\;n\in \mathbb{Z}\}\text{.}
\end{equation*}%
Let $g\in \mathcal{S}_{0}(\mathbb{R})$ and $f\in {L}_{\nu }^{2}(0,1)$. As in
the classical case, we consider the analysis operator 
\begin{equation}
C_{g}^{\nu }(f)=\left( \left\langle f,M_{\beta n}g\right\rangle _{L^{2}(%
\mathbb{R})}\right) _{n}=\left( \left\langle f,\Sigma _{\nu }\left( M_{\beta
n}g\right) \right\rangle _{{L}_{\nu }^{2}(0,1)}\right) _{n}\text{.}
\end{equation}%
The formal adjoint operator of $C_{g}^{\nu }$ is the synthesis operator $%
D_{g}^{\nu }$, which is given by 
\begin{equation}
D_{g}^{\nu }((c)_{n})=\sum_{n\in \mathbb{Z}}c_{n}\Sigma _{\nu }M_{\beta n}g%
\text{.}  \label{Dg}
\end{equation}%
If both operators converge in ${L}_{\nu }^{2}(0,1)$, then the Gabor frame
operator in ${L}_{\nu }^{2}(0,1)$ is%
\begin{equation}
S_{g,\gamma }^{\nu }f=D_{\gamma }^{\nu }C_{g}^{\nu }(f)=\sum_{z\in
Z}\left\langle f,\Sigma _{\nu }M_{\beta n}g\right\rangle _{{L}_{\nu
}^{2}(0,1)}\Sigma _{\nu }(M_{\beta n}\gamma )\text{,}  \label{FrameOp1}
\end{equation}%
\ \ where the dual window $\gamma $ is defined as%
\begin{equation*}
\gamma :=S^{-1}g
\end{equation*}%
and the reconstruction formula $f=S_{g,\gamma }^{\nu }f$ becomes: 
\begin{equation}
f=\sum_{n\in \mathbb{Z}}\left\langle f,\Sigma _{\nu }\left( M_{\beta
n}g\right) \right\rangle _{{L}_{\nu }^{2}(0,1)}\Sigma _{\nu }\left( M_{\beta
n}\gamma \right) \text{.}  \label{reconstruction}
\end{equation}

We will see, in Proposition \ref{prop-conv} below, that $g,\gamma \in 
\mathcal{S}_{0}$ assures that these operators are well defined, in
particular they converge in $L^{2}$ sense. For instance, we recall that the
Feichtinger algebra \cite{Fei} can be defined as follows 
\begin{equation*}
\mathcal{S}_{0}:=\left\{ g\in {L}^{2}(\mathbb{R}):\mathcal{B}g\in \mathcal{F}%
^{1}(\mathbb{C})\right\} \text{.}
\end{equation*}%
The following proposition assures the existence of $D_{g}^{\nu }$ and $%
C_{g}^{\nu }$, whenever $g\in \mathcal{S}_{0}$.

\begin{proposition}
\label{prop-conv} For $g\in \mathcal{S}_{0}$ we have, for every $f\in {L}%
_{\nu }^{2}(0,1)$, 
\begin{equation}
\sum_{n\in \mathbb{Z}}\left\vert \left\langle f,\Sigma _{\nu }\left(
M_{\beta n}g\right) \right\rangle _{{L}_{\nu }^{2}(0,1)}\right\vert
^{2}=\sum_{n\in \mathbb{Z}}\left\vert \left\langle f,M_{\beta
n}g\right\rangle _{L^{2}(\mathbb{R})}\right\vert ^{2}\leq C\Vert g\Vert _{%
\mathcal{S}_{0}}^{2}\Vert f\Vert _{{L}_{\nu }^{2}(0,1)}^{2}
\end{equation}%
and 
\begin{equation}
\left\Vert \sum_{n\in \mathbb{Z}}c_{n}\Sigma _{\nu }M_{\beta n}g\right\Vert
_{{L}_{\nu }^{2}(0,1)}\leq C^{\prime }\Vert g\Vert _{\mathcal{S}_{0}}\Vert
(c_{n})_{n}\Vert _{\ell ^{2}(\mathbb{Z})}.
\end{equation}
\end{proposition}

\begin{proof}
Firstly, by definition, $g\in \mathcal{S}_{0}$ is equivalent to saying that
the function $G=\mathcal{B}g\in \mathcal{F}^{1}(\mathbb{C})$. Secondly, if $%
F $ is an entire function, then the estimate 
\begin{equation*}
\left\vert F(w)e^{-\frac{\pi }{2}\mid w\mid ^{2}}\right\vert \leqslant \frac{%
C}{r^{2}}\int_{D(w,r)}\left\vert F(z)e^{-\frac{\pi }{2}|z|^{2}}\right\vert dz
\end{equation*}%
holds. Thus, 
\begin{equation}
\left\vert \tau _{i\beta n}F(z)\right\vert e^{-\frac{\pi }{2}%
|z|^{2}}=\left\vert F(z-i\beta n)\right\vert e^{-\frac{\pi }{2}|z-i\beta
n|^{2}}\leq C\int_{D(z-i\beta n,r)}\left\vert F(z)e^{-\frac{\pi }{2}%
|z|^{2}}\right\vert dz\text{,}
\end{equation}%
where $\tau _{a}f(z)=\tau (a)f(z)=e^{\pi z\overline{a}-\frac{\pi }{2}%
|a|^{2}}f(z-a)$ is the Weyl operator. Therefore, for $F\in \mathcal{F}%
_{1,\nu }(\mathbb{C}/\mathbb{Z})$, we have%
\begin{equation*}
\sum_{n\in \mathbb{Z}}\left\vert \left\langle f,\Sigma _{\nu }\left(
M_{\beta n}g\right) \right\rangle _{{L}_{\nu }^{2}(0,1)}\right\vert
^{2}=\sum_{n\in \mathbb{Z}}\left\vert \left\langle f,M_{\beta
n}g\right\rangle _{L^{2}(\mathbb{R})}\right\vert ^{2}
\end{equation*}%
and, applying the Bargmann transform as an unitary operator $\mathcal{B}%
:L^{1}(\mathbb{R})\rightarrow \mathcal{F}^{1}(\mathbb{C})$ and using the
intertwining property 
\begin{equation}
\mathcal{B}(\pi (z)g)=e^{i\pi \xi \eta }\tau _{\overline{z}}\mathcal{B}%
g;\quad z\in \mathbb{C}\text{,}  \label{barg-identity}
\end{equation}%
gives%
\begin{equation*}
\sum_{n\in \mathbb{Z}}\left\vert \left\langle f,M_{\beta n}g\right\rangle
_{L^{2}(\mathbb{R})}\right\vert ^{2}=\sum_{n\in \mathbb{Z}}\left\vert
\left\langle F,\tau _{-i\beta n}G\right\rangle \right\vert \text{,}
\end{equation*}%
where $G=\mathcal{B}g,F=\mathcal{B}f\in \mathcal{F}^{1}(\mathbb{C})$. Now, 
\begin{align*}
\sum_{n\in \mathbb{Z}}\left\vert \left\langle F,\tau _{-i\beta
n}G\right\rangle \right\vert & =\sum_{n\in \mathbb{Z}}\left\vert
\left\langle \tau _{i\beta n}F,G\right\rangle \right\vert \\
& \leq C\sum_{n\in \mathbb{Z}}\int_{\mathbb{C}}\int_{\mathbb{C}/\mathbb{Z}%
}\left\vert F(w)\right\vert \chi _{D(z-i\beta n,r)}(w)e^{-\frac{\pi }{2}%
|w|^{2}}dA(w)|G(z)|e^{-\frac{\pi }{2}|z|^{2}}dz \\
& \leq C\sum_{n\in \mathbb{Z}}\int_{\mathbb{C}/\mathbb{Z}}\left\vert
F(w)\right\vert e^{-\frac{\pi }{2}|w|^{2}}\int_{\mathbb{C}}\chi _{D(w+i\beta
n,r)}(z)|G(z)|e^{-\frac{\pi }{2}|z|^{2}}dzdw \\
& \leq C\Vert F\Vert _{\mathcal{F}_{1,\nu }(\mathbb{C}/\mathbb{Z})}\Vert
G\Vert _{1}\text{.}
\end{align*}

If $F\in \mathcal{F}_{\infty ,\nu }(\mathbb{C}/\mathbb{Z})$, then 
\begin{equation}
\left\vert \left\langle \tau _{i\beta n}F,G\right\rangle \right\vert \leq
\Vert F\Vert _{\mathcal{F}_{\infty ,\nu }(\mathbb{C}/\mathbb{Z})}\Vert
G\Vert _{1}
\end{equation}%
and this implies%
\begin{equation*}
\Vert \left\langle F,\tau _{-i\beta n}G\right\rangle \Vert _{\ell ^{\infty }(%
\mathbb{Z})}\leq \Vert F\Vert _{\mathcal{F}_{\infty ,\nu }(\mathbb{C}/%
\mathbb{Z})}\Vert G\Vert _{1}\text{.}
\end{equation*}%
By complex interpolation between $p=1$ and $p=\infty $, we obtain for $F\in 
\mathcal{F}_{2,\nu }(\mathbb{C}/\mathbb{Z})$, 
\begin{equation}
\Vert \left\langle F,\tau _{-i\beta n}G\right\rangle \Vert _{2}\leq C\Vert
F\Vert _{\mathcal{F}_{2,\nu }(\mathbb{C}/\mathbb{Z})}\Vert G\Vert _{1}\text{.%
}
\end{equation}%
Now, since the Bargmann transform $\mathcal{B}$ is an unitary isomorphism
from ${L}_{\nu }^{2}(0,1)$ onto $\mathcal{F}_{2,\nu }(\mathbb{C}/\mathbb{Z})$
and using the intertwining property (\ref{barg-identity}),\ we obtain the
first result. The second result is obtained in the same way. In fact for $%
c=(c_{n})\in \ell ^{1}(\mathbb{Z})$, we have 
\begin{align*}
\left\Vert \sum_{k,n\in \mathbb{Z}}c_{n}e^{2i\pi \nu k-i\pi \beta nk}\tau
_{k-i\beta n}G\right\Vert _{\mathcal{F}_{1,\nu }(\mathbb{C}/\mathbb{Z})}&
\leq \sum_{n\in \mathbb{Z}}|c_{n}|\sum_{k\in \mathbb{Z}}\int_{\mathbb{C}%
/\Gamma }|G(z-k+i\beta n)|e^{-\frac{\pi }{2}|z-k+i\beta n|^{2}}dz \\
& \leq \Vert c\Vert _{\ell ^{1}(\mathbb{Z})}\Vert G\Vert _{\mathcal{F}_{1}(%
\mathbb{C})}\text{.}
\end{align*}%
If $c=(c_{n})\in \ell ^{\infty }(\mathbb{Z})$, 
\begin{align*}
\left\vert \sum_{k,n\in \mathbb{Z}}c_{n}e^{2i\pi \nu k-i\pi \beta nk}\tau
_{k-i\beta n}G(z)\right\vert & \leq \sum_{n,k\in \mathbb{Z}%
}|c_{n}||G(z-k+i\beta n)|e^{-\frac{\pi }{2}|z-k+i\beta n|^{2}} \\
& \leq C\sum_{n,k\in \mathbb{Z}}|c_{n}|\int_{D(z-k+i\beta n,r)}|G(w)|e^{-%
\frac{\pi }{2}|w|^{2}}dw \\
& \leq C\Vert c\Vert _{\ell ^{\infty }(\mathbb{Z})}\Vert G\Vert _{\mathcal{F}%
_{1}(\mathbb{C})}\text{.}
\end{align*}%
Thus, 
\begin{equation*}
\left\Vert \sum_{k,n\in \mathbb{Z}}c_{n}e^{2i\pi \nu k-i\pi \beta nk}\tau
_{k-i\beta n}G\right\Vert _{\mathcal{F}_{\infty ,\nu }(\mathbb{C}/\mathbb{Z}%
)}\leq C\Vert c\Vert _{\ell ^{\infty }(\mathbb{Z})}\Vert G\Vert _{\mathcal{F}%
^{1}(\mathbb{C})}\text{.}
\end{equation*}%
By complex interpolation between $p=1$ and $p=\infty $, we obtain, for $%
c=(c_{n})\in \ell ^{2}(\mathbb{Z})$, 
\begin{equation*}
\left\Vert \sum_{k,n\in \mathbb{Z}}c_{n}e^{2i\pi \nu k-i\pi \beta nk}\tau
_{k-i\beta n}G\right\Vert _{\mathcal{F}_{2,\nu }(\mathbb{C}/\mathbb{Z})}\leq
C\Vert c\Vert _{\ell ^{2}(\mathbb{Z})}\Vert G\Vert _{\mathcal{F}^{1}(\mathbb{%
C})}\text{.}
\end{equation*}%
Now using the identity \eqref{barg-identity} and the unitary isomorphism
property of the Bargmann transform from ${L}_{\nu }^{2}(0,1)$ onto $\mathcal{%
F}_{2,\nu }(\mathbb{C}/\mathbb{Z})$, we obtain the second result.
\end{proof}

\subsection{Janssen's representation and sufficient conditions for Gabor
frames in ${L}_{\protect\nu }^{2}(0,1)$}

By Proposition \ref{prop-conv} the series in \eqref{FrameOp1} defining the
operator $S_{g,\gamma }^{\nu }$ converges in $L^{2}$ sense. Proposition \ref%
{FrameOperator} tells that the frame operator (\ref{FrameOp1}) coincides
with the restriction to ${L}_{\nu }^{2}(0,1)$ of the frame operator
associated with the classical Gabor system $\mathcal{G}\left( g,\Lambda
_{1,\beta }\right) $, where $\Lambda _{1,\beta }=\mathbb{Z}\times \beta 
\mathbb{Z}$. The proof is a simple application of identity (\ref{ident}),
which assumes the following form for $\mathcal{G}_{\nu }\left( g,Z_{\beta
}\right) $:%
\begin{equation}
\left\langle f,\Sigma _{\nu }M_{\beta n}g\right\rangle _{{L}_{\nu
}^{2}(0,1)}=\left\langle f,M_{\beta n}g\right\rangle _{L^{2}\left( \mathbb{R}%
\right) }\text{.}  \label{id}
\end{equation}

\begin{proposition}
\label{FrameOperator}Let $g,\gamma \in \mathcal{S}_{0}$. For $f\in {L}_{\nu
}^{2}(0,1)$, the frame operator $Sf$ defined by (\ref{FrameOp1}), can be
written as 
\begin{equation*}
S_{g,\gamma }^{\nu }f=\sum_{k\in \mathbb{Z}}\sum_{n\in \mathbb{Z}%
}\left\langle f,M_{\beta n}T_{k}g\right\rangle _{L^{2}\left( \mathbb{R}%
\right) }M_{\beta n}T_{k}\gamma \text{.}
\end{equation*}
\end{proposition}

\begin{proof}
If $g,\gamma \in \mathcal{S}_{0}$ then, by Proposition\ \ref{Proposition0}, $%
\Sigma _{\nu }g,\Sigma _{\nu }\gamma \in {L}_{\nu }^{2}(0,1)$ and\ by
Proposition\ (\ref{FrameOp1}), $S_{g,\gamma }$ is bounded.\ From (\ref%
{FrameOp1}) and using the quasi-periodicity condition (\ref{periodic}), $%
f=e^{-2\pi ik\nu }T_{-k}f$, (\ref{id}) gives: 
\begin{eqnarray*}
S_{g,\gamma }^{\nu }f &=&\sum_{n\in Z}\left\langle f,\Sigma _{\nu }M_{\beta
n}g\right\rangle _{{L}_{\nu }^{2}(0,1)}\Sigma _{\nu }(M_{\beta n}\gamma ) \\
&=&\sum_{n\in \mathbb{Z}}\left\langle f,M_{\beta n}g\right\rangle
_{L^{2}\left( \mathbb{R}\right) }\Sigma _{\nu }(M_{\beta n}\gamma ) \\
&=&\sum_{k\in \mathbb{Z}}\sum_{n\in \mathbb{Z}}\left\langle e^{-2\pi ik\nu
}T_{-k}f,M_{\beta n}g\right\rangle _{L^{2}\left( \mathbb{R}\right) }e^{2\pi
ik\nu }T_{k}M_{\beta n}\gamma \\
&=&\sum_{k\in \mathbb{Z}}\sum_{n\in \mathbb{Z}}\left\langle f,T_{k}M_{\beta
n}g\right\rangle _{L^{2}\left( \mathbb{R}\right) }T_{k}M_{\beta n}\gamma 
\text{.}
\end{eqnarray*}
\end{proof}

The simple observation of Proposition \ref{FrameOperator} has the remarkable
consequence that Gabor frames $\mathcal{G}_{\nu }\left( g,Z_{\beta }\right) $
in ${L}_{\nu }^{2}(0,1)$\ inherit some of the structure properties of Gabor
frames in $L^{2}\left( \mathbb{R}\right) $. We will now explore some of
them, omitting the corresponding details when the algebraic manipulations
required are the same as those in the $L^{2}\left( \mathbb{R}\right) $ case.

\begin{itemize}
\item Starting from Proposition \ref{FrameOperator}, one can mimic all the
manipulations in \cite[Section 6.3]{Charly} and arrive at \emph{Walnut's
representation} of the Gabor frame operator. Let $g,\gamma \in \mathcal{S}%
_{0}$ and $f\in {L}_{\nu }^{2}(0,1)$. Then 
\begin{equation*}
S_{g,\gamma }^{\nu }f=\beta ^{-1}\sum_{n\in \mathbb{Z}}G_{n}T_{\frac{n}{%
\beta }}f\text{, }
\end{equation*}%
where $G_{n}$ is called the correlation function of the pair $\left(
g,\gamma \right) $, defined as 
\begin{equation*}
G_{n}(x)=\sum_{k\in \mathbb{Z}}\overline{g}\left( x-\frac{n}{\beta }%
-k\right) \gamma \left( x-k\right) \text{.}
\end{equation*}%
The correlation function is invariant under integer translations and,
similarly as in \cite[(7.7)]{Charly}, we can expand it as a Fourier series%
\begin{equation*}
G_{n}(x)=\sum_{k\in \mathbb{Z}}\left\langle \gamma ,M_{k}T_{\frac{n}{\beta }%
}g\right\rangle e^{2\pi ikx}\text{.}
\end{equation*}

\item This leads, as in \cite[Section 7.2]{Charly} to the important \emph{%
Janssen's representation }of the Gabor frame operator. Given $f\in {L}_{\nu
}^{2}(0,1)$, then 
\begin{equation*}
S_{g,\gamma }^{\nu }f=\beta ^{-1}\sum_{k\in \mathbb{Z}}\sum_{n\in \mathbb{Z}%
}\left\langle \gamma ,M_{k}T_{\frac{n}{\beta }}g\right\rangle _{L^{2}\left( 
\mathbb{R}\right) }M_{k}T_{\frac{n}{\beta }}f
\end{equation*}%
Janssen's representation is one of the fundamental results in time-frequency
analysis. It has several applications and notorious ramifications to other
mathematical corners, for instance, it can be interpreted as a quantum theta
function and as a Poisson formula for the sympletic Fourier transform \cite%
{LuefManin}. For our purposes, the most important consequence is the
following version of the Wexler-Raz biorthogonality relations.
\end{itemize}

\begin{corollary}
\label{WexlerRaz} Suppose that $g,\gamma \in \mathcal{S}_{0}$. If $%
S_{g,\gamma }=I$ or, equivalently, if there exists $\gamma \in \mathcal{S}%
_{0}$ such that 
\begin{equation}
\beta ^{-1}\sum_{n\in \mathbb{Z}}\left\langle \gamma ,M_{k}T_{\frac{n}{\beta 
}}g\right\rangle _{L^{2}\left( \mathbb{R}\right) }e^{2i\pi \frac{nl}{\beta }%
}=\delta _{k,0}\text{ \ for all }l\in \mathbb{Z}\text{,}  \label{wrp}
\end{equation}%
then $\mathcal{G}_{\nu }\left( g,Z_{\beta }\right) $ is Gabor frame for ${L}%
_{\nu }^{2}(0,1)$.
\end{corollary}

\begin{remark}
The original Wexler-Raz relations (\ref{WR}) provide both a necessary and
sufficient condition for Gabor frames in $L^{2}\left( \mathbb{R}\right) $,
while the ${L}_{\nu }^{2}(0,1)$ analogue given by (\ref{wrp})\ is only a
sufficient condition, since we couldn't show that the Gabor frame property
of $\mathcal{G}_{\nu }\left( g,Z_{\beta }\right) $ in ${L}_{\nu }^{2}(0,1)$
assures the existence of the dual window $\gamma $ in (\ref{wrp}).
\end{remark}

\begin{remark}
\label{Method}The Wexler-Raz biorthogonality relations (\ref{WR}) clearly
imply (\ref{wrp}). However, this requires finding a function $\gamma \in 
\mathcal{S}_{0}$ such that $V_{g}\gamma (0,0)=\beta $ and $V_{g}\gamma (k,%
\frac{n}{\beta })=0$, $k,n\in \mathbb{N-\{}0\mathbb{\}}$, while (\ref{wrp})
only requires vanishing in the periodization lattice. Corollary \ref%
{WexlerRaz} assures that $\mathcal{G}_{\nu }\left( g,Z_{\beta }\right) $ is
a frame indexed by the sampling lattice $Z_{\beta }=i\beta \mathbb{Z}$, once
we find a function interpolating on the periodization lattice $\mathbb{Z-\{}0%
\mathbb{\}}$.
\end{remark}

\begin{remark}
Since%
\begin{equation*}
V_{g}\Sigma _{\frac{nl}{\beta }}\gamma (x,\frac{n}{\beta })=\sum_{n\in 
\mathbb{Z}}\left\langle \gamma ,M_{x}T_{\frac{n}{\beta }}g\right\rangle
_{L^{2}\left( \mathbb{R}\right) }e^{2i\pi \frac{nl}{\beta }}\text{,}
\end{equation*}%
to use (\ref{wrp}) one only needs to find a function $\gamma \in \mathcal{S}%
_{0}$ such that\ for all $l\in \mathbb{Z}$, $V_{g}\Sigma _{\frac{nl}{\beta }%
}\gamma $ vanishes on $\mathbb{Z-\{}0\mathbb{\}}$ and such that $V_{g}\Sigma
_{\frac{nl}{\beta }}\gamma (0,\frac{n}{\beta })=\beta $. Moreover, since $%
\gamma \in \mathcal{S}_{0}$, we can assure that $\Sigma _{\frac{nl}{\beta }%
}\gamma \in {L}_{\nu }^{2}(0,1)$ and $V_{g}\Sigma _{\frac{nl}{\beta }}\gamma
\in \mathcal{H}_{g}^{\frac{nl}{\beta }}\left( \mathbb{R}^{2}/\mathbb{Z}%
\right) $. This means that $V_{g}\Sigma _{\frac{nl}{\beta }}\gamma $
satisfies periodicity conditions with respect to $\Gamma ^{\prime }=i\frac{1%
}{\beta }\mathbb{Z}$, with character given by $\chi _{l}(\gamma )=\chi _{l}(i%
\frac{n}{\beta })=e^{2i\pi \frac{nl}{\beta }}$. In section \ref{GaborHermite}
we will implement this method to show our results about Gabor systems with
Hermite-theta functions. The key idea is to find a function $H\in \mathcal{F}%
_{1,l}(\mathbb{C}/\Gamma ^{\prime })$ vanishing on $\mathbb{Z-\{}0\mathbb{\}}
$ and use the surjectivity of the periodization operator (\ref{per}) from $%
\mathcal{F}_{1}(\mathbb{C})$ into $\mathcal{F}_{1,l}(\mathbb{C}/\Gamma
^{\prime })$, to assure the existence of $F\in \mathcal{F}_{1,l}(\mathbb{C}%
/\Gamma ^{\prime })$ such that $H=\mathcal{P}_{\Gamma ^{\prime },l}(F)$;
this is equivalent to the existence of a dual window $\gamma =\mathcal{B}%
^{-1}F\in \mathcal{S}_{0}$ such that (\ref{wrp}) holds.
\end{remark}

\begin{corollary}
\label{LR_L(0,1)} Suppose that $g,\gamma \in \mathcal{S}_{0}$. If $\mathcal{G%
}\left( g,\Lambda _{1,\beta }\right) $ is a Gabor frame for $L^{2}\left( 
\mathbb{R}\right) $, then $\mathcal{G}_{\nu }\left( g,Z_{\beta }\right) $ is
a Gabor frame for ${L}_{\nu }^{2}(0,1)$.
\end{corollary}

\begin{proof}
If $\mathcal{G}\left( g,\Lambda _{1,\beta }\right) $ is a Gabor frame for $%
L^{2}\left( \mathbb{R}\right) $, and $g,\gamma \in \mathcal{S}_{0}$, then $%
D_{g}$ and $D_{\gamma }$\ are bounded. Thus, the Wexler-Raz biorthogonality
relations (\ref{WR}), 
\begin{equation}
\beta ^{-1}\left\langle \gamma ,M_{k}T_{\frac{n}{\beta }}g\right\rangle
_{L^{2}\left( \mathbb{R}\right) }=\delta _{n,0}\delta _{k,0}\text{ \ for all 
}k,n\in \mathbb{Z}\text{,}  \label{wr}
\end{equation}%
hold \cite[Theorem 7.3.1]{Charly}. Since $g,\gamma \in \mathcal{S}_{0}$
then, by Proposition\ \ref{Proposition0}, $\Sigma _{\nu }g$,$\Sigma _{\nu
}\gamma \in {L}_{\nu }^{2}(0,1)$ and\ by Proposition\ (\ref{FrameOp1}), $%
S_{g,\gamma }$ is bounded.\ This implies (\ref{wrp}) and, by Corollary \ref%
{WexlerRaz}, $\mathcal{G}_{\nu }\left( g,Z_{\beta }\right) $ is a Gabor
frame for ${L}_{\nu }^{2}(0,1)$.
\end{proof}

From this, we easily derive a \emph{new criteria for obstructions for Gabor
frames}.

\begin{corollary}
Let $g,\gamma \in \mathcal{S}_{0}$. If there exists $f\in {L}_{\nu
}^{2}(0,1) $ such that $V_{g}f(0,\beta n)=0$, $n\in \mathbb{N}$\ , then $%
\mathcal{G}\left( g,\Lambda _{1,\beta }\right) $ \emph{is not} a Gabor frame
for $L^{2}\left( \mathbb{R}\right) $.
\end{corollary}

\begin{proof}
If $g\in \mathcal{S}_{0}$, By Proposition\ \ref{Proposition0}, $\Sigma _{\nu
}M_{\beta n}g\in {L}_{\nu }^{2}(0,1)$. Thus, if there exists $f\in {L}_{\nu
}^{2}(0,1)$ such that%
\begin{equation*}
\left\langle f,\Sigma _{\nu }M_{\beta n}g\right\rangle _{{L}_{\nu
}^{2}(0,1)}=\left\langle f,M_{\beta n}g\right\rangle _{L^{2}\left( \mathbb{R}%
\right) }=V_{g}f(0,\beta n)=0\text{, \ }n\in \mathbb{N}\text{,}
\end{equation*}%
then, by the inequalities defining the sampling condition (\ref{id}), $%
\mathcal{G}_{\nu }\left( g,Z_{\beta }\right) $ is not a Gabor frame for ${L}%
_{\nu }^{2}(0,1)$. By Corollary \ref{LR_L(0,1)}, $\mathcal{G}\left(
g,\Lambda _{1,\beta }\right) $ is not a Gabor frame for $L^{2}\left( \mathbb{%
R}\right) $.
\end{proof}

We will not explore this in the present paper, but\ since the frame
condition for $\mathcal{G}_{\nu }\left( g,Z_{\beta }\right) $\ depends only
on the size of the parameter $\beta $, it is reasonable to expect situations
where one is able to disprove a frame property for a system $\mathcal{G}%
\left( g,\Lambda _{1,\beta }\right) $.\ 

Corollary \ref{LR_L(0,1)} allows to transfer sufficient conditions about
Gabor frames in $L^{2}\left( \mathbb{R}\right) $, of the form $\mathcal{G}%
\left( g,\Lambda _{1,\beta }\right) $, to $\mathcal{G}_{\nu }\left(
g,Z_{\beta }\right) $, a Gabor frame for ${L}_{\nu }^{2}(0,1)$.\ A chief
example is the family of\emph{\ totally positive functions}.\ By a theorem
of Schoenberg~\cite{sch51}, the Fourier transform of an integrable totally
positive function\ $g$ possesses the factorization 
\begin{equation}
\hat{g}(\xi )=ce^{-\gamma \xi ^{2}}e^{2\pi i\nu \xi }\prod_{j=1}^{N}(1+2\pi
i\nu _{j}\xi )^{-1}e^{-2\pi i\nu _{j}\xi }  \label{eq:t1}
\end{equation}%
with $c>0$, $\nu ,\nu _{j}\in \mathbb{R}$, $\gamma \geq 0$, $N\in \mathbb{N}%
\cup \{\infty \}$ and $0<\gamma +\sum_{j}\nu _{j}^{2}<\infty $. Special
cases, when $N$ is finite, have been considered before in \cite%
{Groe2013,Groe2018} and the full problem has been recently solved for
rational lattices by Gr\"{o}chenig in \cite{GroAIM}. Despite the incremental
nature of the sequence of work \cite{Groe2013,Groe2018,GroAIM}, each of
these papers depends on completely different methods and should be
considered as a whole. In particular, \cite{Gro2018} has an aesthetically
pleasing flavour, since it deals with this intrinsic difficult problem with
a surprising virtuose simplicity. Totally positive functions (\ref{eq:t1})
include Gaussian windows and many other examples, in particular the case of
the hyperbolic secant $\left( e^{at}+e^{-at}\right) ^{-1}$, where the
factorization (\ref{eq:t1}) is infinite, previously proved in \cite{HypSec}.
The strongest result about Gabor frames with totally positive functions,
which includes the case of an infinite factorization ($N=\infty $)\ is the
following recent Theorem.

\textbf{Theorem }\label{GroAIM}\textbf{\ (}Gr\"{o}chenig\textbf{\ }\cite%
{GroAIM}\textbf{). }Let $g\in L^{1}(\mathbb{R})\cap L^{2}(\mathbb{R})$ be a
totally positive function other than the one sided exponential $e^{-t}1_{%
\mathbb{R}^{+}}(t)$, and assume that $\alpha \beta $ is \emph{rational}.
Then, if $\alpha \beta <1$, then $\mathcal{G}\left( g,\Lambda _{\alpha
,\beta }\right) $ is a Gabor frame for $L^{2}\left( \mathbb{R}\right) $.

Combining this with Corollary \ref{LR_L(0,1)}, we obtain:

\begin{corollary}
\label{FrameTotal}Let $g\in L^{1}(\mathbb{R})\cap \mathcal{S}_{0}(\mathbb{R}%
) $ be a totally positive function other than the one sided exponential $%
e^{-t}1_{\mathbb{R}^{+}}(t)$, and assume that $\beta $ is \emph{rational}.
Then, if $\beta <1$, $\mathcal{G}_{\nu }\left( g,Z_{\beta }\right) $ is a
Gabor frame for ${L}_{\nu }^{2}(0,1)$.
\end{corollary}

\begin{proof}
If $\beta <1$, by Theorem \ref{GroAIM}, $\mathcal{G}\left( g,\Lambda
_{\alpha ,\beta }\right) $ is a Gabor frame for $L^{2}\left( \mathbb{R}%
\right) $. This, together with the assumption $g\in \mathcal{S}_{0}(\mathbb{R%
})$ and $\beta $ rational, allows to use Theorem 13.2.1 in \cite[pg. 279]%
{Charly} (originally proved in \cite[Thm. 3.4]{FeiGro}) to conclude that the
frame operator $S_{gg}^{1,\beta }$ is invertible on $\mathcal{S}_{0}(\mathbb{%
R})$ and that the canonical dual window $\gamma ^{0}=\left( S_{gg}^{\alpha
,\beta }\right) ^{-1}g$ belongs to $\mathcal{S}_{0}(\mathbb{R})$. From
Corollary \ref{LR_L(0,1)}, we conclude that $\mathcal{G}_{\nu }\left(
g,Z_{\beta }\right) $ is Gabor frame for ${L}_{\nu }^{2}(0,1)$.
\end{proof}

We restrict here to rational lattices, but there are also results that
exclusively hold for irrational lattices (this is a more structured case,
since the group von Neumann algebra becomes a factor\cite{Omland} ) case of
irrational latttices, which are believed to provide frames for any function
sampled above the critical density. This belief was materialized in the case
of Hurwitz functions in \cite{Belov}, where it is shown that, despite many
rational obstructions, if the latttice is irrational, then if $\alpha \beta
<1$, yields Gabor frames for $L^{2}\left( \mathbb{R}\right) $.

\subsection{\label{GaborHermite}Gabor frames with Hermite-theta windows}

When $\Lambda $ is a $2d$-lattice, sufficient conditions for $\mathcal{G}%
\left( h_{r},\Lambda \right) $ to be a frame in ${L}^{2}(\mathbb{R})$, have
been obtained in \cite{CharlyYura} (fully describing the lattices yielding
frames $\mathcal{G}\left( h_{r},\Lambda \right) $ remains a problem of
ongoing research interest \cite{FauIriIlya,NonFrameHermite}). For the
vectorial system (super-frame) $\mathcal{G}\left( (h_{0},...h_{N-1}),\Lambda
\right) $ in ${L}^{2}(\mathbb{R},\mathbb{C}^{N})$, a full description of $2d$%
-lattices generating a super-frame was obtained in \cite{CharlyYurasuper}.
In \cite{Abr2010}, this has been identified as a sampling result for the
Fock space of polyanalytic functions, $\mathbf{F}^{(N)}\left( \mathbb{C}%
\right) $, and the corresponding interpolation result has also been proved
using the description of multi-sampling and interpolating in $\mathcal{F}%
^{(0)}\left( \mathbb{C}\right) $ has been given, together with the
corresponding result for super Riesz sequences, and this has been shown to
correspond to sampling and interpolation results in $\mathbf{F}_{\nu
}^{(n)}\left( \mathbb{C}\right) $.

By simple computation using Poisson summation formula as well as $\mathcal{F}%
h_{r}=i^{r}h_{r}$, we get%
\begin{equation}
\Sigma _{\nu }(M_{\beta n}h_{r})(t)=i^{r}\sum_{k\in \mathbb{Z}}e^{-\pi
(k-\beta n+\nu )^{2}}H_{r}\left( \sqrt{2\pi }(k-\beta n+\nu )\right)
e^{2i\pi (k+\nu )t}\text{,}
\end{equation}%
where 
\begin{equation*}
H_{r}(t)=\frac{2^{1/4}}{\sqrt{r!}}\left( -\frac{1}{\sqrt{2}}\right)
^{r}e^{t^{2}}\frac{d^{r}}{dt^{r}}\left( e^{-t^{2}}\right)
\end{equation*}%
are the Hermite polynomials. We will use the notation $\theta _{\alpha
,\beta }^{(r)}$ for the Hermite theta functions 
\begin{equation}
\theta _{\alpha ,\beta }^{(r)}(z):=\sum_{k\in \mathbb{Z}}e^{-\pi (k+\alpha
)^{2}+2i\pi (z-\beta )(k+\alpha )}H_{r}\left( \sqrt{2\pi }(k+\alpha +\Im
(z)\right) \text{.}  \label{Hermite_Theta}
\end{equation}%
to write the Gabor system $\mathcal{G}_{\nu }\left( h_{r},Z_{\beta }\right) $%
\ as follows: 
\begin{equation*}
\mathcal{G}_{\nu }\left( h_{r},Z_{\beta }\right) =\{e^{2i\pi \beta
nt}i^{r}\theta _{\nu -\beta n,0}^{(r)}(t)\;|\;n\in \mathbb{Z}\}\subset {L}%
_{\nu }^{2}(0,1)\text{.}
\end{equation*}%
In this section we prove the following result.

\begin{theorem}
\label{FrameHermite}If $\beta <\frac{1}{r+1}$, then $\mathcal{G}_{\nu
}\left( h_{r},Z_{\beta }\right) $ is a Gabor frame for ${L}_{\nu }^{2}(0,1)$.
\end{theorem}

One could use the result of Gr\"{o}chenig and Lyubarskii \cite{CharlyYura},
where it is proved that, if $\alpha \beta <1$, then $\mathcal{G}_{\nu
}\left( h_{r},\Lambda _{\alpha ,\beta }\right) $ is a Gabor frame for $%
L^{2}\left( \mathbb{R}\right) $, together with Corollary \ref{LR_L(0,1)},
Theorem 13.2.1 in \cite[pg. 279]{Charly} and Theorem 1.2 in \cite{GL}. We
give a direct, self-contained proof that doesn't require the results from 
\cite{CharlyYura} and exemplifies how to apply the method described in
Remark \ref{Method}, based on the condition (\ref{wrp}).

We have been assuming, without loss of generality, that $\Gamma =\mathbb{Z}$%
. But for the results in this section, it is worth to use the flexibility
offered by introduzing a multiplying parameter $\alpha $ and consider $%
\Gamma =\mathbb{Z}\alpha $. Then the signal domain is the interval $%
(0,\alpha )$ and the flat cylinder phase space is $\Lambda (\mathbb{Z}\alpha
):\mathbb{=}[0,\alpha )\times \mathbb{R}$. A change of variables can
transform the signal domain in $(-\alpha /2,\alpha /2)$ and the flat
cylinder into the more symmetric alternative $\Lambda ^{\prime }((\mathbb{Z+}%
\frac{1}{2})\alpha ):\mathbb{=}[-\alpha /2,\alpha /2)\times \mathbb{R}$.
Gabor analysis of signals defined in intervals of the form $(-\alpha
/2,\alpha /2)$ has been recently considered for the problem of uniqueness of
STFT phase retrieval \cite{LukPhi}. The same problem of uniqueness of STFT
phase retrieval has been considered, for functions bandlimited to $\left[
-2B,2B\right] $, in \cite{AlaiWell}, leading to the horizontal flat cylinder
domain $\mathbb{R}\times \left[ -2B,2B\right] $ as phase space. This
suggests, as a possible variant of the theory developed in this paper, a
STFT analysis of band-limited functions, resulting from imposing periodic
conditions in frequency and properly adapting the scheme of this paper.
Despite not directly dealing with problems of phase retrieval, the natural
occurence of flat cylinder STFT phase spaces in such problems was one of our
motivations. We leave as a question for the interested reader if our
sampling and interpolation results and the implied statements about sets of
uniqueness can replace the role of the completeness results for exponential
systems and Paley-Wiener spaces used in the proofs of \cite{LukPhi}.

\subsection{Proof of Theorem \protect\ref{FrameHermite}}

Consider the lattice $\Gamma ^{\prime }=i\frac{1}{\beta }\mathbb{Z}$\ and\
the Banach space $\mathcal{F}_{1,l}(\mathbb{C}/\Gamma ^{\prime })$ 
\begin{equation*}
\mathcal{F}_{1,l}(\mathbb{C}/\Gamma ^{\prime }):=\left\{ f\in Hol(\mathbb{C}%
)\ \text{satisfying (\ref{funct-equa-l}):}\quad \;\int_{\mathbb{C}/\Gamma
^{\prime }}|f(z)|e^{-\frac{\pi }{2}|z|^{2}}<+\infty \right\} \text{.}
\end{equation*}%
To use Corollary \ref{WexlerRaz}, we need a function $\gamma \in \mathcal{S}%
_{0}$ such that 
\begin{equation}
\beta ^{-1}\sum_{n\in \mathbb{Z}}\left\langle \gamma ,M_{k}T_{\frac{n}{\beta 
}}h_{r}\right\rangle _{L^{2}\left( \mathbb{R}\right) }e^{2i\pi \frac{nl}{%
\beta }}=\delta _{k,0}\text{ \ for all }l\in \mathbb{Z}\text{.}
\label{rel-int}
\end{equation}%
First observe that, for\ $\gamma \in \mathcal{S}_{0}$ and, 
\begin{equation}
\left\langle \gamma ,\pi (\overline{z})h_{r}\right\rangle _{{L}^{2}(\mathbb{R%
})}=\frac{1}{\sqrt{\pi ^{r}r!}}e^{i\pi x\xi -\frac{\pi }{2}\left\vert
z\right\vert ^{2}}\left( \partial _{z}-\pi \overline{z}\right)
^{r}F(z),\quad z=x+i\xi \text{.}  \label{formula-hermite-win}
\end{equation}%
We get that the relation (\ref{rel-int}) is equivalent to 
\begin{equation*}
\frac{1}{\beta ^{-1}\sqrt{\pi ^{r}r!}}\sum_{n\in \mathbb{Z}}e^{2i\pi \frac{nl%
}{\beta }}e^{-i\pi k\frac{n}{\beta }-\frac{\pi }{2}\left\vert k-i\frac{n}{%
\beta }\right\vert ^{2}}\left( \partial _{z}-\pi \overline{z}\right)
^{r}F(z)|_{(z=z+\gamma )}=\delta _{k,0}\text{ \ for all }l\in \mathbb{Z}%
\text{.}
\end{equation*}%
This suggests to find a periodized function 
\begin{equation}
H(z)=\mathcal{P}_{\Gamma ^{\prime },l}(F)(z)=\sum_{\gamma \in \Gamma
^{\prime }}\chi _{l}(\gamma )e^{-\pi z\overline{\gamma }-\frac{\pi }{2}%
|\gamma |^{2}}F(z+\gamma )\text{;}\quad \Gamma ^{\prime }=i\frac{1}{\beta }%
\mathbb{Z}\text{.}  \label{H}
\end{equation}%
Since the raising operator $\partial _{z}-\pi \bar{z}$ commutes with the
Weyl operator $W_{a}(F)(z)=e^{-\pi z\bar{a}-\frac{\pi }{2}|a|^{2}}F(z+a)$,
applying it $r$ times to both sides of the equation gives 
\begin{equation}
\left( \partial _{z}-\pi \overline{z}\right) ^{r}H(z)=\sum_{\gamma \in
\Gamma ^{\prime }}\chi _{l}(\gamma )e^{-\pi z\overline{\gamma }-\frac{\pi }{2%
}|\gamma |^{2}}\left( \partial _{z}-\pi \overline{z}\right)
^{r}F(z)|_{z=z+\gamma }\text{;}\quad \Gamma ^{\prime }=i\frac{1}{\beta }%
\mathbb{Z}\text{.}
\end{equation}%
Taking $z=k$ and multiplying by the constant $\frac{1}{\beta ^{-1}\sqrt{\pi
^{r}r!}}$, we see that the function $H$ must satisfy 
\begin{equation}
\frac{1}{\beta ^{-1}\sqrt{\pi ^{r}r!}}\left( \partial _{z}-\pi \overline{z}%
\right) ^{r}H(z)|_{z=k}=\delta _{k,0}\quad \text{for all }l\in \mathbb{Z}%
\text{.}  \label{eqi-blue}
\end{equation}%
Let $\chi _{l}$ be the character given by $\chi _{l}(\gamma )=\chi _{l}(i%
\frac{n}{\beta })=e^{2i\pi \frac{nl}{\beta }}$. We need a holomorphic
function $H_{\beta }(z)$ satisfying the functional equation 
\begin{equation}
H_{\beta }(z+\gamma )=\chi _{l}(\gamma )e^{\frac{\pi }{2}|\gamma |^{2}+\pi z%
\overline{\gamma }}H_{\beta }(z)  \label{funct-equa-l}
\end{equation}%
with respect to the lattice $\Gamma ^{\prime }=i\frac{1}{\beta }\mathbb{Z}$,
vanishing on $\mathbb{Z}\setminus \{0\}$, and with every zero of $H_{\beta
}(z)$ of multiplicity $r+1$. One can verify from (\ref{H}), that any
periodized function $H(z)=\mathcal{P}_{\Gamma ^{\prime },l}(F)(z)$ satisfies
(\ref{funct-equa-l}). Now, set%
\begin{equation*}
g_{\beta }(z)=\prod_{k>0}(1-e^{2\beta \pi (z-k)})\times
\prod_{k<0}(1-e^{2\beta \pi (k-z)})\text{.}
\end{equation*}%
The product converges uniformly on every compact subset of $\mathbb{C}$,
thus $g_{\beta }$ is a holomorphic function on $\mathbb{C}$. Furthermore,
this function vanishes on $\mathbb{Z}-\{0\}$ and is periodic with respect to 
$\Gamma ^{\prime }$. Using \cite[Lemma 6]{ALZ}, we conclude that $g_{\beta }$
satisfies 
\begin{equation*}
\left\vert g_{\beta }(z)\right\vert \leq C_{1}e^{\pi \beta |x|^{2}+\gamma
^{+}|x|};\quad z=x+iy
\end{equation*}%
and 
\begin{equation}
\left\vert g_{\beta }(z)\right\vert \geq C_{1}e^{\pi \beta |x|^{2}-\gamma
^{-}|x|}d(z,\mathbb{Z})\text{.}  \label{lower-bd}
\end{equation}%
Now, we consider 
\begin{equation}
H_{\beta }(z)=e^{-\frac{\pi }{2}z^{2}}\left( g_{\beta }(z)\right) ^{r+1}%
\text{.}  \label{Hbeta}
\end{equation}%
Then $H_{\beta }(z)$ satisfies (\ref{funct-equa-l}) and vanishes on $\mathbb{%
Z}-\{0\}$. Moreover, 
\begin{align*}
\left\Vert H_{\beta }\right\Vert _{\mathcal{F}_{1,l}(\mathbb{C}/\Gamma
^{\prime })}& =C\left\Vert H_{\beta }\right\Vert _{\mathcal{F}_{1,l}(\mathbb{%
C}/\Gamma ^{\prime })}=\int_{\mathbb{C}/\Gamma ^{\prime }}|g_{\beta
}(z)|^{r+1}e^{-\pi |x|^{2}}dA(z) \\
& \leq C\int_{\mathbb{R}\times \lbrack 0,\frac{1}{\beta }]}e^{\pi (r+1)\beta
|x|^{2}+\gamma ^{+}|x|-\pi |x|^{2}}dxdy\text{.}
\end{align*}%
Since $\beta <\frac{1}{r+1}$, we see that this function belongs to $\mathcal{%
F}_{1,l}(\mathbb{C}/\Gamma ^{\prime })$. Now, Lemma \ref%
{ComplexPeriodization} assures the existence of a function $F\in \mathcal{F}%
_{1}(\mathbb{C})$ such that 
\begin{equation*}
H_{\beta }(z)=\mathcal{P}_{\Gamma ^{\prime },l}(F)(z)\text{.}
\end{equation*}%
We conclude that (\ref{rel-int}) holds with $\gamma =(\mathcal{B}%
^{(r)})^{-1}F\in \mathcal{S}_{0}$. This completes the proof of the theorem.

\section{Beurling density in the flat cylinder}

A sequence $Z=\{z_{k}\}\subset \Lambda (\Gamma )$ is\emph{\ separated} if 
\begin{equation*}
\delta :=\inf_{j\neq k}\left\vert z_{j}-z_{k}\right\vert >0\text{.}
\end{equation*}%
The constant $\delta $ is called the \emph{separation constant} of $Z$. A
separated sequence $Z$ is a \emph{sampling sequence} for the space $\mathcal{%
H}_{g}^{\nu }\left( \mathbb{C}/\mathbb{Z}\right) $, whenever there exist
constants $A,B>0$ such that, for all $f\in {L}_{\nu }^{2}(0,1)$\ the
sampling inequality holds:%
\begin{equation}
A\left\Vert V_{g}f\right\Vert _{\mathcal{H}_{g}^{\nu }\left( \mathbb{C}/%
\mathbb{Z}\right) }^{2}\leq \sum_{z\in Z}\left\vert V_{g}f(z)\right\vert
^{2}\leq B\left\Vert V_{g}f\right\Vert _{\mathcal{H}_{g}^{\nu }\left( 
\mathbb{C}/\mathbb{Z}\right) }^{2}\text{.}  \label{samplingGabor}
\end{equation}%
We have seen in the introductory discussion of this section that this is
equivalent to $\mathcal{G}_{\nu }\left( g,Z\right) :=\{\Sigma _{\nu }\left(
\pi (z)g\right) ,\;z\in Z\}$ being a \emph{Gabor frame }for ${L}_{\nu
}^{2}(0,1)$, according to (\ref{frame}). For the choices $g(t)=h_{r}$, using
the unitarity of the Bargmann-type transforms discussed in Section 3, (\ref%
{samplingGabor}) can be written as a sampling inequality for $\mathcal{F}%
_{\nu }^{(r)}(\mathbb{C}/\mathbb{Z})$: 
\begin{equation}
A\left\Vert F\right\Vert _{\mathcal{F}_{\nu }^{(r)}(\mathbb{C}/\mathbb{Z}%
)}^{2}\leq \sum_{z\in Z}\left\vert F(z)\right\vert ^{2}e^{-\pi |z|^{2}}\leq
B\left\Vert F\right\Vert _{\mathcal{F}_{\nu }^{(r)}(\mathbb{C}/\mathbb{Z}%
)}^{2}\text{.}  \label{sampling}
\end{equation}%
Now we will define \emph{Gabor Riesz basic sequences} as the sequences
leading to \emph{interpolating sequences for the range of the STFT} on ${L}%
_{\nu }^{2}(0,1)$. This differs a bit from the most common definition, but
it is equivalent to the concept of Gabor Riesz basis of their span,
considered in \cite{Charly}.

\begin{definition}
A sequence $Z$ is said to be \emph{interpolating} for $\mathcal{H}_{g}^{\nu
}\left( \mathbb{C}/\mathbb{Z}\right) $, if there exists a constant $C$ such,
for every sequence $(a_{k})_{k\in \mathbb{Z}}$ such that%
\begin{equation*}
\left( a_{k}\right) _{k\in \mathbb{Z}}\in \ell ^{2}(\mathbb{Z})\text{,}
\end{equation*}%
one can find an interpolating function $f\in \mathcal{H}_{g}^{\nu }\left( 
\mathbb{C}/\mathbb{Z}\right) $, such that$\ f(z_{k})=a_{k}$ for all $%
z_{k}\in Z$ and%
\begin{equation}
\Vert f\Vert _{\mathcal{H}_{g}^{\nu }\left( \mathbb{C}/\mathbb{Z}\right)
}\leq C\Vert f|Z\Vert _{2,\alpha }\text{.}  \label{stab}
\end{equation}%
If $Z$ is interpolating for $\mathcal{H}_{g}^{\nu }\left( \mathbb{C}/\mathbb{%
Z}\right) $, then $\mathcal{G}_{\nu }\left( g,Z\right) $ is called a \emph{%
Gabor} \emph{Riesz basic sequence} for ${L}_{\nu }^{2}(0,1)$.
\end{definition}

In \cite{ALZ},\ we were able to obtain a complete description of the
sampling and interpolating sequences in the space of entire functions $%
\mathcal{F}_{\nu }(\mathbb{C}/\mathbb{Z})$, after introducing the concepts
of upper and lower Beurling density for the vertical strip. The upper
density concept will be used again in the result about interpolation in the
last section. We take $\Lambda (\mathbb{Z})=[0,1]\times \mathbb{R}$ as a
fundamental domain representing the group $\mathbb{C}/\mathbb{Z}$. Write $%
I_{w,r}=[0,1]\times \lbrack \Im (w)-\frac{r}{2},\Im (w)+\frac{r}{2}]$. Then $%
\lim_{r\rightarrow +\infty }I_{w,r}=[0,1]\times \mathbb{R}$. For a given
sequence $Z$ of distinct numbers in $\Lambda (\mathbb{Z})$, let $%
n(Z,I_{w,r}) $ denote the number of points in $Z\cap I_{w,r}$. The lower and
upper Beurling densities of $Z$ are given, respectively, by 
\begin{equation*}
D^{-}(Z)=\liminf_{r\rightarrow +\infty }\inf_{w\in \mathbb{C}/\mathbb{Z}}%
\frac{n(Z,I_{w,r})}{r}
\end{equation*}%
and 
\begin{equation*}
D^{+}(Z)=\limsup_{r\rightarrow +\infty }\sup_{w\in \mathbb{C}/\mathbb{Z}}%
\frac{n(Z,I_{w,r})}{r}\text{.}
\end{equation*}

\begin{example}
\label{regular}When $Z$ is the regular lattice $Z_{\beta }=i\beta \mathbb{Z}$%
, $D^{-}(Z_{\beta })=D^{+}(Z_{\beta })=\frac{1}{\beta }$.
\end{example}

\subsection{Gabor frames with Gaussian-theta windows}

For the choice $g(t)=h_{0}=e^{-\pi t^{2}}$ we have a complete description of
the sequences yielding frames in ${L}_{\nu }^{2}(0,1)$. From (\ref{BargDef}%
), 
\begin{equation*}
\mathcal{B}f(z)=e^{-i\pi x\xi +\frac{\pi }{2}\left\vert z\right\vert
^{2}}V_{g}(\overline{z})=e^{-i\pi x\xi +\frac{\pi }{2}\left\vert
z\right\vert ^{2}}\left\langle f,\pi (\overline{z})g\right\rangle _{{L}^{2}(%
\mathbb{R})}\text{.}
\end{equation*}%
Thus, since $\mathcal{B}:{L}_{\nu }^{2}(0,1)\rightarrow \mathcal{F}_{\nu }(%
\mathbb{C}/\mathbb{Z})$ is unitary, given $F\in \mathcal{F}_{\nu }(\mathbb{C}%
/\mathbb{Z})$, there exists $f\in {L}_{\nu }^{2}(0,1)$ such that $F(z)=%
\mathcal{B}f(z)$. We observe that%
\begin{equation*}
\Sigma _{\nu }(\pi (z_{1},z_{2})h_{0})(t)=e^{2i\pi z_{2}t}\theta
_{z_{2},z_{1}}(-t,i)
\end{equation*}%
and that (\ref{sampling}) is equivalent to $\mathcal{G}_{\nu }\left(
h_{0},Z\right) $ being a Gabor frame for ${L}_{\nu }^{2}(0,1)$. The
following result is proved in \cite{ALZ}:

\begin{theorem}
\label{cct} $Z$ is sampling for $\mathcal{F}_{\nu }^{2}(\mathbb{C}/\mathbb{Z}%
)$, or, equivalently, $\mathcal{G}_{\nu }\left( h_{0},Z\right) =\{e^{2i\pi
z_{2}t}\theta _{z_{2},z_{1}}(-t,i):\quad z=(z_{1},z_{2})\in Z\}$ is a \emph{%
Gabor frame} for ${L}_{\nu }^{2}(0,1)$, iff $D^{-}(Z)>1$, while $Z$ is
interpolating for $\mathcal{F}_{\nu }^{2}(\mathbb{C}/\mathbb{Z})$ iff $%
D^{+}(Z)<1$.
\end{theorem}

From Example \ref{regular}, when $Z=Z_{\beta }=i\beta \mathbb{Z}$, $\mathcal{%
G}_{\nu }\left( h_{0},Z_{\beta }\right) =\{e^{2i\pi \beta t}\theta _{\beta
,0}(-t,i)\}$. In this special case, $\mathcal{G}_{\nu }\left( h_{0},Z_{\beta
}\right) $ is a Gabor frame iff $\frac{1}{\beta }>1$ and a Gabor Riesz
sequence iff $\frac{1}{\beta }<1$.\ 

As a step in the proof of Theorem \ref{cct}, the following explicit sampling
formula which may be used for signal reconstruction as an alternative to
frame expansions, has been obtained.

\begin{proposition}
Let $Z=\{z_{n}\}$ be a sequence uniformly close to $Z_{\beta }$, and assume $%
\frac{1}{\beta }>1$. Let 
\begin{equation*}
g_{\beta }(z)=e^{\frac{\pi }{2}z^{2}}\prod_{k\geq 0}(1-e^{2i\pi
(z_{k}-z)})\times \prod_{k\leq 0}(1-e^{2i\pi (z-z_{k})}).
\end{equation*}
Then, for $f\in \mathcal{F}_{\nu }^{\infty }\left( \mathbb{C}/\mathbb{Z}%
\right) $, we have 
\begin{equation*}
f(z)=2i\pi \sum_{k\in \mathbb{Z}}f(z_{k})\frac{g_{\beta }(z)}{g_{\beta
}^{\prime }(z_{k})(1-e^{2i\pi (z_{k}-z)})}\text{,}
\end{equation*}%
with uniform convergence on compacts of $\mathbb{C}/\Gamma $.
\end{proposition}

\subsection{Interpolating sequences for $\mathcal{F}_{\protect\nu %
}^{(r)}\left( \mathbb{C}/\mathbb{Z}\right) $}

Our next result is a sufficient condition for interpolating in $\mathcal{F}%
_{\nu }^{(r)}\left( \mathbb{C}/\mathbb{Z}\right) $.

\begin{theorem}
\label{Interpolation}Suppose that $Z$ is a separated sequence. If $D^{+}(Z)<%
\frac{1}{r+1}$, then $Z$ is interpolating for $\mathcal{F}_{\nu
}^{(r)}\left( \mathbb{C}/\mathbb{Z}\right) $.
\end{theorem}

\begin{proof}
As in the proof of sufficiency of interpolation for the analytic case in 
\cite[Theorem 6]{ALZ}, we may therefore assume that $Z$ is uniformly close
to the sequence $\Lambda _{\beta }=i\frac{1}{\beta }\mathbb{Z}$, and that $%
\beta <\frac{1}{r+1}$. Consider the function 
\begin{equation*}
g_{n,Z}:=g_{n}(z)=\left\{ 
\begin{matrix}
\prod_{k\neq n}(1-e^{2i\pi (z_{k}-z)})\times \prod_{k<0}(1-e^{2i\pi
(z-z_{k})})\text{,} & \text{ if }n\geq 0 \\ 
\prod_{k\geq n}(1-e^{2i\pi (z_{k}-z)})\times \prod_{k\neq n}(1-e^{2i\pi
(z-z_{k})})\text{,} & \text{ if }n<0%
\end{matrix}%
\right. \text{.}
\end{equation*}%
Next, let%
\begin{equation*}
G_{n,Z}(z):=G_{n}(z)=(1-e^{2i\pi (z_{n}-z)})^{r}\left( g_{n,Z}\right) ^{r+1}%
\text{.}
\end{equation*}%
Taking into account the general form of a function in $\mathcal{F}_{\nu
}^{(r)}\left( \mathbb{C}/\mathbb{Z}\right) $ given by (\ref%
{TruePolyRepresentation}), and that 
\begin{equation*}
e^{\pi \left\vert z\right\vert ^{2}}\left( \partial _{z}\right) ^{r}\left[
e^{-\pi \left\vert z\right\vert ^{2}}\left( e^{\frac{\pi }{2}%
z^{2}+Az}G(z)\right) \right] =e^{\frac{\pi }{2}z^{2}+Az}\left( \partial
_{z}+2i\pi \Im (z)+A\right) ^{r}G(z)\text{,}
\end{equation*}%
given $z_{n}=x_{n}+iy_{n}$, write $w_{n}=x_{n}+i\lfloor y_{n}\rfloor $ and
consider the following interpolation function 
\begin{equation}
f(z)=C_{r}\sum_{n\in \mathbb{Z}}a_{n}\frac{e^{\pi (z-z_{n})\overline{w_{n}}+%
\frac{\pi }{2}\left( (z-w_{n})^{2}-(z_{n}-w_{n})^{2}\right) }}{%
(g_{n}(z_{n}))^{r-1}}\left( \partial _{z}+2i\pi (\Im (z)-\lfloor
y_{n}\rfloor )\right) ^{r}G_{n}(z)\text{,}  \label{func-interpolation}
\end{equation}%
where $C_{r}=\left( \frac{\pi ^{r}}{(r)!}\right) ^{\frac{1}{2}}\left( \frac{1%
}{(2i\pi )^{r}}\right) $. Clearly, $f(z_{k})=a_{k}$ for all $k\in \mathbb{Z}$%
. To show that $f\in \mathcal{F}_{\nu }^{(r)}\left( \mathbb{C}/\mathbb{Z}%
\right) $, define the auxiliary function 
\begin{equation}
F(z)=\sum_{n\in \mathbb{Z}}a_{n}\frac{e^{\pi (z-z_{n})\overline{w_{n}}+\frac{%
\pi }{2}\left( (z-w_{n})^{2}-(z_{n}-w_{n})^{2}\right) }}{(g_{n}(z_{n}))^{r}}%
G_{n}(z)(z)\text{.}
\end{equation}%
By the same arguments as in \cite[Theorem 6]{ALZ}, we have $F\in \mathcal{F}%
_{\nu }\left( \mathbb{C}/\mathbb{Z}\right) $. The result then follows by
observing that 
\begin{equation*}
f(z)=C_{r}\left( \partial _{z}-\pi \overline{z}\right) ^{r}F(z)
\end{equation*}%
and that $\left( \partial _{z}-\pi \overline{z}\right) ^{r}$ is an isometric
isomorphism maping $\mathcal{F}_{\nu }\left( \mathbb{C}/\mathbb{Z}\right) $
onto $\mathcal{F}_{\nu }^{(r)}\left( \mathbb{C}/\mathbb{Z}\right) $.
\end{proof}

When $Z$ is the regular lattice $Z=i\beta \mathbb{Z}$, $D^{+}(Z)=\frac{1}{%
\beta }$. Thus, according to the above theorem, if $\beta >r+1$ then $i\beta 
\mathbb{Z}$ is interpolating for $\mathcal{F}_{\nu }^{(r)}\left( \mathbb{C}/%
\mathbb{Z}\right) $.

\section{Polyanalytic Fock spaces on $\mathbb{C}/\mathbb{Z}$}

We conclude this paper by defining the polyanalytic Fock spaces $\mathbf{F}%
_{\nu }^{(n)}\left( \mathbb{C}/\mathbb{Z}\right) $, since it seems that they
have not been considered before, or at least we could not find them in the
literature, in particular, \cite%
{DimensionTheta,GhanmiIntissar2008,GhanmiIntissar2013,IZ1,IZ2,Z1} only deal
with the true (Landau levels) eigenspaces $\mathcal{F}_{\nu }^{(r)}\left( 
\mathbb{C}/\mathbb{Z}\right) $, considered in Section 3. In \cite[Remark 1.5]%
{Ghanmi2023}, the possibility of the construction of polyanalytic Fock space
is pointed out, but the idea is not explored further. \emph{The polyanalytic
Fock spaces }$\mathbf{F}_{\nu }^{(n)}\left( \mathbb{C}/\mathbb{Z}\right) $%
\emph{\ should not be confused with the true polyanalytic spaces, or Landau
levels eigenspaces }$\mathcal{F}_{\nu }^{(r)}\left( \mathbb{C}/\mathbb{Z}%
\right) $, considered in Section 3, but we will show that $\mathbf{F}_{\nu
}^{(n)}\left( \mathbb{C}/\mathbb{Z}\right) $ can be decomposed as a direct
sum of the first $n$ spaces $\mathcal{F}_{\nu }^{(r)}\left( \mathbb{C}/%
\mathbb{Z}\right) $. This extends to the $\mathbb{C}/\mathbb{Z}$ setting
Vasilevski's orthogonal decomposition of $\mathbf{F}^{(n)}\left( \mathbb{C}%
\right) $ into true polyanalytic Fock spaces $\mathcal{F}^{(r)}\left( 
\mathbb{C}\right) $. Moreover, the spaces $\mathbf{F}_{\nu }^{(n)}\left( 
\mathbb{C}/\mathbb{Z}\right) $ will be linked to a vectorial STFT acting on $%
{L}_{\nu }^{2}\left[ (0,1),\mathbb{C}^{N}\right] $, when the vector of
windows is composed of Hermite functions, extending to $\mathbf{F}_{\nu
}^{(n)}\left( \mathbb{C}/\mathbb{Z}\right) $ a similar connection known for $%
\mathbf{F}^{(n)}\left( \mathbb{C}\right) $ \cite{Abr2010,AbrGabor}.

\subsection{Vector-valued STFT on ${L}_{\protect\nu }^{2}\left[ (0,1),%
\mathbb{C}^{N}\right] $}

A construction of a vectorial STFT acting on $L^{2}\left( \mathbb{R},\mathbb{%
C}^{N}\right) $ appeared in \cite{Abr2010,AbrGabor}, motivated by Balan's
idea of using Gabor superframes for the purpose of multiplexing of signals
(sending several signals in the same channel and later separate them at the
receiver) \cite{balan00,BalanWH}. See also \cite{HanLar} for a deep
theoretical study of superframes. For the choice of the vectorial window $%
\mathbf{g=h}_{N}=(h_{0},...,h_{N-1})$, the vectorial STFT acting on $%
L^{2}\left( \mathbb{R},\mathbb{C}^{N}\right) $, becomes a multiple of a
Bargmann-type integral transform mapping $L^{2}\left( \mathbb{R},\mathbb{C}%
^{N}\right) $ to the Fock space of polyanalytic functions in $\mathbb{C}$,
introduced in the proof of the results in \cite{Abr2010}.

In this subsection we describe an analogue construction of a vectorial STFT
acting on ${L}_{\nu }^{2}\left[ (0,1),\mathbb{C}^{N}\right] $. It will be
used in Section 3.3 to define a Bargmann-type transform mapping to the
(full) Fock space of polyanalytic functions on $\mathbb{C}/\mathbb{Z}$. Let $%
{L}_{\nu }^{2}\left[ (0,1),\mathbb{C}^{N}\right] $ be the vectorial Hilbert
space defined by the inner product below, where $\mathbf{f=}\left(
f_{0},...,f_{N-1}\right) $ and $\mathbf{g=}\left( g_{0},...,g_{N-1}\right) $%
: 
\begin{equation}
\left\langle \mathbf{f},\mathbf{g}\right\rangle _{{L}_{\nu }^{2}\left[ (0,1),%
\mathbb{C}^{N}\right] }=\sum_{i=0}^{N-1}\left\langle
f_{i},g_{i}\right\rangle _{{L}_{\nu }^{2}(0,1)}\text{.}
\label{scalarProduct}
\end{equation}%
Given a vectorial window $\mathbf{g=}\left( g_{0},...,g_{N-1}\right) \in
L^{2}\left( \mathbb{R},\mathbb{C}^{N}\right) $,%
\begin{equation}
\left\langle g_{i},g_{j}\right\rangle =\delta _{i,j}\text{, }i,j=0,...,N-1%
\text{,}  \label{ort}
\end{equation}%
one can define the vectorial STFT acting on $\mathbf{f=}\left(
f_{0},...,f_{N-1}\right) \in {L}_{\nu }^{2}\left[ (0,1),\mathbb{C}^{N}\right]
$, by 
\begin{equation}
\mathbf{V}_{\mathbf{g}}\mathbf{f}(x,\xi
):=\sum_{i=0}^{N-1}V_{g_{i}}f_{i}(x,\xi )\text{.}  \label{vectorGabor}
\end{equation}%
This defines an isometry%
\begin{equation*}
\mathbf{V}_{\mathbf{g}}:{L}_{\nu }^{2}\left[ (0,1),\mathbb{C}^{N}\right]
\rightarrow L^{2}(\mathbb{C}/\mathbb{Z})
\end{equation*}%
since, from (\ref{Moyal}) and (\ref{ort}),%
\begin{equation}
\left\Vert \mathbf{V}_{\mathbf{g}}\mathbf{f}\right\Vert _{L^{2}\left( 
\mathbb{C}/\mathbb{Z}\right) }^{2}=\left\langle
\sum_{i=0}^{N-1}V_{g_{i}}f_{i}(x,\xi ),\sum_{i=0}^{N-1}V_{g_{i}}f_{i}(x,\xi
)\right\rangle _{L^{2}([0,1]\times \mathbb{R})}=\left\Vert \mathbf{f}%
\right\Vert _{{L}_{\nu }^{2}\left[ (0,1),\mathbb{C}^{N}\right] }^{2}\text{.}
\label{isometry}
\end{equation}%
We define the super Gabor space 
\begin{equation}
\mathbf{H}_{\mathbf{g}}^{\nu }\left( \mathbb{C}/\mathbb{Z}\right) :=\mathbf{V%
}_{\mathbf{g}}\left( {L}_{\nu }^{2}\left[ (0,1),\mathbb{C}^{N}\right]
\right) \text{.}  \label{superGabor}
\end{equation}

\subsection{Polyanalytic Fock spaces $\mathbf{F}_{\protect\nu }^{(N)}\left( 
\mathbb{C}/\mathbb{Z}\right) $}

Consider $\mathbf{F}_{\nu }^{(N)}\left( \mathbb{C}\right) $, the Fock space
of polyanalytic functions, that is, satisfying the generalized
Cauchy-Riemann equation%
\begin{equation*}
\partial _{\overline{z}}^{N}F(z)=0\text{,}
\end{equation*}%
and belonging to $L^{2}(\mathbb{C},e^{-\pi |z|^{2}})$. A fundamental result
in the theory of polyanalytic Fock spaces is Vasilevski's orthogonal
decomposition of $\mathbf{F}^{(n)}\left( \mathbb{C}\right) $ into true
polyanalytic Fock spaces $\mathcal{F}^{(r)}\left( \mathbb{C}\right) $ \cite%
{VasiFock}, namely,%
\begin{equation}
\mathbf{F}^{(n)}\left( \mathbb{C}\right) =\bigoplus\limits_{r=0}^{N-1}%
\mathcal{F}^{(r)}\left( \mathbb{C}\right) \text{.}  \label{VasiFock}
\end{equation}%
Therefore, it follows from (\ref{repTrue}) and from the well-known identity
for Laguerre polynomials:%
\begin{equation*}
L_{n}^{1}(x)=\sum_{r=0}^{N-1}L_{r}(x)\text{,}
\end{equation*}%
that the reproducing kernel of $\mathbf{F}^{(n)}\left( \mathbb{C}\right) $ is%
\begin{equation*}
K_{\mathbf{F}^{(n)}\left( \mathbb{C}\right) }(z,w)=\sum_{r=0}^{N-1}K_{%
\mathcal{F}^{(r)}\left( \mathbb{C}\right) }(z,w)=e^{\pi \overline{w}%
z}L_{n}^{1}(\pi \left\vert z-w\right\vert ^{2})\text{.}
\end{equation*}%
Clearly, $\cup _{n=0}^{\infty }\mathbf{F}^{(n)}\left( \mathbb{C}\right)
=L^{2}(\mathbb{C},e^{-\pi |z|^{2}})$, thus we also have%
\begin{equation}
L^{2}(\mathbb{C},e^{-\pi |z|^{2}})=\bigoplus\limits_{r=0}^{\infty }\mathcal{F%
}^{(r)}\left( \mathbb{C}\right) \text{.}  \label{VasiL2}
\end{equation}%
Now consider $\mathbf{F}_{\nu }^{(N)}\left( \mathbb{C}/\mathbb{Z}\right) $,
the Fock space of polyanalytic functions belonging to $L^{2}(\mathbb{C}/%
\mathbb{Z},e^{-\pi |z|^{2}})$ and satisfying the functional equation%
\begin{equation*}
F(z+k)=e^{2\pi ik\nu }e^{\pi kz+\frac{\pi }{2}\left\vert k\right\vert
^{2}}F(z)\text{.}
\end{equation*}%
The next result computes the reproducing kernel of $\mathbf{F}_{\nu
}^{(N)}\left( \mathbb{C}/\mathbb{Z}\right) $ and extends Vasilevski's
decompositions (\ref{VasiFock}) and (\ref{VasiL2}) to $\mathbf{F}_{\nu
}^{(N)}\left( \mathbb{C}/\mathbb{Z}\right) $. Our proof is different from
the one in \cite{VasiFock}, where restriction operators involving Fourier
transforms have been used (maybe the approach of \cite{VasiFock} also works,
but we didn't try it). Instead, we compute the reproducing kernel of $%
\mathbf{F}_{\nu }^{(N)}\left( \mathbb{C}/\mathbb{Z}\right) $ and decompose
it as the sum of the first $n$ reproducing kernels of $\mathcal{F}_{\nu
}^{(r)}\left( \mathbb{C}/\mathbb{Z}\right) $. The orthogonality of the
spaces $\mathcal{F}_{\nu }^{(r)}\left( \mathbb{C}/\mathbb{Z}\right) $ for
any pair of natural numbers $(r_{1},r_{2})$ with $r_{1}\neq r_{2}$, is then
an immediate consequence of (\ref{ortTruepoly}). This is an example of how
the time-frequency formulation can simplify some proofs, since the
orthogonality of the spaces $\mathcal{F}_{\nu }^{(r)}\left( \mathbb{C}/%
\mathbb{Z}\right) $, would otherwise require the evaluation of the integrals
using Green's formula and integration by parts (the same comment applies to
the double indexed orthogonality (\ref{doubleindexort}) of the sequence $%
\left\{ \varphi _{k}^{(r)}\right\} _{k,r=0}^{\infty }$, which, under our
approach is an immediate consequence of the Moyal-type formulas (\ref{Moyal}%
).

\begin{theorem}
The reproducing kernel of $\mathbf{F}_{\nu }^{(N)}\left( \mathbb{C}/\mathbb{Z%
}\right) $ is%
\begin{equation*}
K_{\mathbf{F}_{\nu }^{(N)}\left( \mathbb{C}/\mathbb{Z}\right) }(z,w)=e^{\pi 
\overline{w}z}\sum_{k\in \mathbb{Z}}e^{2\pi ik\nu }e^{\frac{\pi }{2}%
k^{2}+\pi (z+\overline{w})k}L_{r}^{1}(\pi \left\vert z-w+k\right\vert
^{2})=\sum_{r=0}^{N-1}K_{\mathcal{F}_{\nu }^{(r)}\left( \mathbb{C}/\mathbb{Z}%
\right) }(z,w)\text{.}
\end{equation*}%
Moreover, the spaces $\mathbf{F}_{\nu }^{(N)}\left( \mathbb{C}/\mathbb{Z}%
\right) $ and $L^{2}(\mathbb{C}/\mathbb{Z},e^{-\pi |z|^{2}})$ have the
following orthogonal decompositions into true polyanalytic spaces $\mathcal{F%
}_{\nu }^{(r)}\left( \mathbb{C}/\mathbb{Z}\right) $:%
\begin{equation}
\mathbf{F}_{\nu }^{(N)}\left( \mathbb{C}/\mathbb{Z}\right)
=\bigoplus\limits_{r=0}^{N-1}\mathcal{F}_{\nu }^{(r)}\left( \mathbb{C}/%
\mathbb{Z}\right)  \label{VasiPeriodicTrue}
\end{equation}%
and%
\begin{equation}
L^{2}(\mathbb{C}/\mathbb{Z},e^{-\pi |z|^{2}})=\bigoplus\limits_{r=0}^{\infty
}\mathcal{F}_{\nu }^{(r)}\left( \mathbb{C}/\mathbb{Z}\right) \text{.}
\label{VasiPeriodicL2}
\end{equation}
\end{theorem}

\begin{proof}
We first proceed as in \cite[Proof of Lemma 9]{GhanmiIntissar2008}, to show
that the following representation formula holds for $F\in \mathbf{F}_{\nu
}^{(N)}\left( \mathbb{C}/\mathbb{Z}\right) $: 
\begin{equation}
F(z)=\int_{\mathbb{C}}F(z)e^{\pi z\overline{w}}L_{n}^{1}(\pi \left\vert
z-w\right\vert ^{2})e^{-\pi |w|^{2}}d(w).  \label{reproducing}
\end{equation}%
First recall that this holds as the reproducing formula for $F\in \mathbf{F}%
_{\nu }^{(N)}\left( \mathbb{C}\right) $. Now, let $F\in \mathbf{F}_{\nu
}^{(N)}\left( \mathbb{C}/\mathbb{Z}\right) $. Then $\partial _{\bar{z}%
}^{n}F=0$ and $|F(z)|\leq Ce^{\frac{\pi }{2}|z|^{2}}$. Set $F_{\epsilon
}(z)=F(\epsilon z)$, for a given $\epsilon \in \left] 0,1\right[ $. Then $%
|F(\epsilon z)|\leq Ce^{\frac{\pi }{2}\epsilon ^{2}|z|^{2}}$ implies%
\begin{equation*}
\int_{\mathbb{C}}|F_{\epsilon }(z)|^{2}e^{-\pi |z|^{2}}dz\leq C\int_{\mathbb{%
C}}e^{-\pi (1-\epsilon ^{2})|z|^{2}}dz<\infty
\end{equation*}%
Thus, since $\partial _{\bar{z}}^{n+1}F_{\epsilon }=0$,\ $F_{\epsilon
}(z)\in \mathbf{F}_{\nu }^{(N)}\left( \mathbb{C}\right) $ and satisfies (\ref%
{reproducing}) for any $\epsilon \in \left] 0,1\right[ $. Letting $\epsilon
\rightarrow 1$, we conclude by the dominate convergence theorem that $F$
satisfies (\ref{reproducing}). Now, let $F\in \mathbf{F}_{\nu }^{(N)}\left( 
\mathbb{C}/\mathbb{Z}\right) $. Using (\ref{reproducing}) and $%
F(z+k)=e^{2\pi ik\nu }e^{\frac{\pi }{2}k^{2}+\pi zk}F(z)$, leads to: 
\begin{eqnarray*}
F(z) &=&\int_{\mathbb{C}}F(w)e^{\pi \overline{w}z}L_{n}^{1}(\pi \left\vert
z-w\right\vert ^{2})e^{-\pi |w|^{2}}dw \\
&=&\sum_{k\in \mathbb{Z}}\int_{(0,1)\times \mathbb{R}}F(w+k)e^{\pi (%
\overline{w}+k)z}L_{n}^{1}(\pi \left\vert z-w-k\right\vert ^{2})e^{-\pi
|w+k|^{2}}dw \\
&=&\sum_{k\in \mathbb{Z}}\int_{(0,1)\times \mathbb{R}}e^{2\pi ik\nu }e^{-%
\frac{\pi }{2}k^{2}-\pi \overline{w}k}F(w)e^{\pi (\overline{w}%
+k)z}L_{n}^{1}(\pi \left\vert z-w-k\right\vert ^{2})e^{-\pi |w|^{2}}dw \\
&=&\int_{(0,1)\times \mathbb{R}}e^{\pi \overline{w}z}\sum_{k\in \mathbb{Z}%
}e^{2\pi ik\nu }e^{-\frac{\pi }{2}k^{2}+\pi (z-\overline{w})k}L_{n}^{1}(\pi
\left\vert z-w-k\right\vert ^{2})e^{-\pi |w|^{2}}dw\text{.}
\end{eqnarray*}%
Thus, the reproducing kernel of $\mathbf{F}_{\nu }^{(N)}\left( \mathbb{C}/%
\mathbb{Z}\right) $ is%
\begin{eqnarray*}
K_{\mathbf{F}_{\nu }^{(N)}\left( \mathbb{C}/\mathbb{Z}\right) }(z,w)
&=&e^{\pi \overline{w}z}\sum_{k\in \mathbb{Z}}e^{2\pi ik\nu }e^{-\frac{\pi }{%
2}k^{2}+\pi (z-\overline{w})k}L_{n}^{1}(\pi \left\vert z-w-k\right\vert ^{2})
\\
&=&e^{\pi \overline{w}z}\sum_{k\in \mathbb{Z}}e^{2\pi ik\nu }e^{-\frac{\pi }{%
2}k^{2}+\pi (z-\overline{w})k}\sum_{r=0}^{N-1}L_{r}(\pi \left\vert
z-w-k\right\vert ^{2})
\end{eqnarray*}%
leading to%
\begin{equation*}
K_{\mathbf{F}_{\nu }^{(N)}\left( \mathbb{C}/\mathbb{Z}\right)
}(z,w)=\sum_{r=0}^{N-1}K_{\mathcal{F}_{\nu }^{(r)}\left( \mathbb{C}/\mathbb{Z%
}\right) }(z,w)\text{,}
\end{equation*}%
so that a generic element $F\in \mathbf{F}_{\nu }^{(N)}\left( \mathbb{C}/%
\mathbb{Z}\right) $ can be written as%
\begin{eqnarray*}
F(w) &=&\int_{\mathbb{C}}F(z)K_{\mathbf{F}_{\nu }^{(N)}\left( \mathbb{C}/%
\mathbb{Z}\right) }(z,w)dz \\
&=&\sum_{r=0}^{N-1}\int_{\mathbb{C}}F(z)K_{\mathcal{F}_{\nu }^{(r)}\left( 
\mathbb{C}/\mathbb{Z}\right) }(z,w)dz \\
&=&\sum_{r=0}^{N-1}(P_{\mathcal{F}_{\nu }^{(r)}\left( \mathbb{C}/\mathbb{Z}%
\right) }F)(w)\text{,}
\end{eqnarray*}%
where $P_{\mathcal{F}_{\nu }^{(r)}\left( \mathbb{C}/\mathbb{Z}\right) }F\in 
\mathcal{F}_{\nu }^{(r)}\left( \mathbb{C}/\mathbb{Z}\right) $ is the
projection of $F$ over $\mathcal{F}_{\nu }^{(r)}\left( \mathbb{C}/\mathbb{Z}%
\right) $. Together with the orthogonality of the spaces $\mathcal{F}_{\nu
}^{(r)}\left( \mathbb{C}/\mathbb{Z}\right) $, which follows from (\ref%
{ortTruepoly}), this establishes (\ref{VasiPeriodicTrue}). Since $\cup
_{n=0}^{\infty }\mathbf{F}_{\nu }^{(N)}\left( \mathbb{C}\right) =L^{2}(%
\mathbb{C},e^{-\pi |z|^{2}})$, this also implies (\ref{VasiPeriodicL2}).
\end{proof}

A Bargmann-type transform%
\begin{equation*}
\mathbf{B}^{\left( n\right) }:{L}_{\nu }^{2}\left[ (0,1),\mathbb{C}^{N}%
\right] \rightarrow \mathbf{F}_{\nu }^{(N)}\left( \mathbb{C}/\mathbb{Z}%
\right)
\end{equation*}%
can be defined, given $\mathbf{f=}\left( f_{0},...,f_{N-1}\right) \in {L}%
_{\nu }^{2}\left[ (0,1),\mathbb{C}^{N}\right] $ and $\mathbf{g=}\left(
g_{0},...,g_{N-1}\right) \in {L}^{2}\left[ \mathbb{R},\mathbb{C}^{N}\right] $%
, by 
\begin{equation}
\mathbf{B}^{\left( N\right) }\mathbf{f}(z)=\sum_{r=0}^{N-1}\mathcal{B}%
^{\left( r\right) }f_{r}(z)\text{.}  \label{SuperBarg}
\end{equation}%
Using (\ref{trupolydef}) and (\ref{vectorGabor}), this can be written as%
\begin{equation*}
\mathbf{B}^{\left( N\right) }f(z)=\sum_{r=0}^{N-1}\mathcal{B}^{\left(
r\right) }f_{r}(z)=e^{-i\pi x\xi +\frac{\pi }{2}\left\vert z\right\vert
^{2}}\sum_{r=0}^{N-1}V_{h_{r}}f_{r}(x,-\xi )=e^{-i\pi x\xi +\frac{\pi }{2}%
\left\vert z\right\vert ^{2}}\mathbf{V}_{\mathbf{h}_{N}}\mathbf{f}(x,-\xi )%
\text{,}
\end{equation*}%
where $\mathbf{h}_{N}\mathbf{=}\left( h_{0},...,h_{N-1}\right) $,\ and it
follows from (\ref{isometry}) that $\mathbf{B}^{\left( N\right) }:{L}_{\nu
}^{2}\left[ (0,1),\mathbb{C}^{N}\right] \rightarrow \mathbf{F}_{\nu
}^{(N)}\left( \mathbb{C}/\mathbb{Z}\right) $ is isometric. This construction
has been recently suggested in \cite[Remark 1.5]{Ghanmi2023}. Fock spaces of
polyanalytic functions of infinite order in $\mathbb{C}$ have been recently
considered in \cite{InfinitePolyanalytic}. The construction seems to be
possible to adapt to $\mathbb{C}/\mathbb{Z}$, but we will not pursue this
direction further.

\subsection{Gabor superframes with Hermite functions}

As in the scalar case, we will define Gabor superframes by considering first
sampling sequences. Thus, we say that $Z\subset \lbrack 0,1)\times \mathbb{R}
$ is sampling, whenever there exist constants $A,B>0$ such that, for all $%
\mathbf{f}\in {L}_{\nu }^{2}\left[ (0,1),\mathbb{C}^{N}\right] $,\ the
inequality holds:%
\begin{equation}
A\left\Vert \mathbf{V}_{\mathbf{h}_{N}}\mathbf{f}\right\Vert _{\mathcal{H}%
_{g}^{\nu }\left( \mathbb{C}/\mathbb{Z}\right) }^{2}\leq \sum_{z\in
Z}\left\vert \mathbf{V}_{\mathbf{h}_{N}}\mathbf{f}(z)\right\vert ^{2}\leq
B\left\Vert \mathbf{V}_{\mathbf{h}_{N}}\mathbf{f}\right\Vert _{\mathcal{H}%
_{g}^{\nu }\left( \mathbb{C}/\mathbb{Z}\right) }^{2}\text{.}
\label{superframe}
\end{equation}%
Let us consider the regular lattice case $Z=Z_{\beta }:=i\beta \mathbb{Z}%
\subset \lbrack 0,1)\times \mathbb{R}$. Equivalently, $Z_{\beta }$ consists
of the pairs $(0,\beta n)\subset \lbrack 0,1)\times \mathbb{R}$, $n\in 
\mathbb{Z}$. Then $\pi (z)\mathbf{g}=M_{\beta n}\mathbf{g}$ and, combining
the definition (\ref{scalarProduct}) with (\ref{ident}), gives 
\begin{equation*}
\left\langle \mathbf{f},\Sigma _{\nu }M_{\beta n}\mathbf{g}\right\rangle _{{L%
}_{\nu }^{2}\left[ (0,1),\mathbb{C}^{N}\right] }=\left\langle \mathbf{f}%
,M_{\beta n}\mathbf{g}\right\rangle _{{L}^{2}\left[ \mathbb{R},\mathbb{C}^{N}%
\right] }\text{.}
\end{equation*}%
Thus, using (\ref{isometry}), the inequalities (\ref{frame}) can be written
as 
\begin{equation}
A\left\Vert \mathbf{f}\right\Vert _{{L}_{\nu }^{2}\left[ (0,1),\mathbb{C}^{N}%
\right] }^{2}\leq \sum_{z\in Z}\left\vert \left\langle \mathbf{f},\Sigma
_{\nu }M_{\beta n}\mathbf{g}\right\rangle _{{L}_{\nu }^{2}\left[ (0,1),%
\mathbb{C}^{N}\right] }\right\vert ^{2}\leq B\left\Vert \mathbf{f}%
\right\Vert _{{L}_{\nu }^{2}\left[ (0,1),\mathbb{C}^{N}\right] }^{2}\text{.}
\label{Superframe}
\end{equation}%
This justifies the following definition of superframe in ${L}_{\nu }^{2}%
\left[ (0,1),\mathbb{C}^{N}\right] $ .

\begin{definition}
Given a vectorial window $\mathbf{g=}\left( g_{0},...,g_{N-1}\right) \in
L^{2}\left( \mathbb{R},\mathbb{C}^{N}\right) $, with the entries satisfying
the orthonormality $\left\langle g_{i},g_{j}\right\rangle =\delta _{i,j}$, $%
i,j=0,...,N-1$, one declares the Gabor system $\mathcal{G}_{\nu }\left( 
\mathbf{g},Z\right) :=\{\Sigma _{\nu }\left( M_{\beta n}\mathbf{g}\right)
,\;n\in \mathbb{Z}\}$ to be a \emph{Gabor superframe} for ${L}_{\nu }^{2}%
\left[ (0,1),\mathbb{C}^{N}\right] $,\emph{\ }whenever there exist constants 
$A,B>0$ such that, for all $f\in {L}_{\nu }^{2}(0,1)$, (\ref{Superframe})
hold.
\end{definition}

For vector-valued systems, the same arguments as in Proposition \ref%
{FrameOperator}, allow to write the Gabor superframe operator in ${L}_{\nu
}^{2}\left[ (0,1),\mathbb{C}^{N}\right] $ as%
\begin{equation*}
\mathbf{S}_{g,\gamma }^{\nu }\mathbf{f}=\sum_{z\in Z}\left\langle \mathbf{f}%
,\Sigma _{\nu }M_{\beta n}\mathbf{g}\right\rangle _{{L}_{\nu }^{2}\left[
(0,1),\mathbb{C}^{N}\right] }\Sigma _{\nu }(M_{\beta n}\mathbf{\gamma })%
\text{,}
\end{equation*}%
\ \ where the dual window $\mathbf{\gamma }$ is defined as%
\begin{equation*}
\mathbf{\gamma }:=\mathbf{S}^{-1}\mathbf{g}\text{.}
\end{equation*}%
Let $g_{i},\gamma _{i}\in \mathcal{S}_{0}$. For $\mathbf{f}\in {L}_{\nu }^{2}%
\left[ (0,1),\mathbb{C}^{N}\right] $, the same manipulations as in
Proposition \ref{FrameOperator}, lead to 
\begin{equation*}
\mathbf{S}_{\mathbf{g,\gamma }}^{\nu }\mathbf{f}=\sum_{k\in \mathbb{Z}%
}\sum_{n\in \mathbb{Z}}\left\langle \mathbf{f},M_{\beta n}T_{k}\mathbf{g}%
\right\rangle _{L^{2}\left( \mathbb{R},\mathbb{C}^{N}\right) }M_{\beta
n}T_{k}\mathbf{\gamma }\text{.}
\end{equation*}%
Let $g_{i},\gamma _{i}\in \mathcal{S}_{0}$, $i=1,...,N$ and $\mathbf{f}\in {L%
}_{\nu }^{2}\left[ (0,1),\mathbb{C}^{N}\right] $. The steps leading to \emph{%
Janssen's representation }of the Gabor frame operator can be repeated as in 
\cite[5-Appendix]{CharlyYurasuper}, and the result is 
\begin{equation*}
\mathbf{S}_{\mathbf{g,\gamma }}^{\nu }\mathbf{f}=\beta ^{-1}\sum_{k\in 
\mathbb{Z}}\sum_{n\in \mathbb{Z}}\left\langle \mathbf{\gamma },M_{k}T_{\frac{%
n}{\beta }}\mathbf{g}\right\rangle _{L^{2}\left( \mathbb{R},\mathbb{C}%
^{N}\right) }M_{k}T_{\frac{n}{\beta }}\mathbf{f}\text{.}
\end{equation*}

\begin{proposition}
\label{Superwr}Let $g_{i},\gamma _{i}\in \mathcal{S}_{0}$, $i=1,...,N$. If $%
\mathbf{S}_{g,\gamma }=\mathbf{I}$ or, equivalently, if there exists $%
\mathbf{\gamma =(}\gamma \mathbf{_{1},...,}\gamma \mathbf{\mathbf{_{N}})}$
such that $\gamma _{i}\in \mathcal{S}_{0}$, $i=1,...,N$, 
\begin{equation}
\beta ^{-1}\sum_{n\in \mathbb{Z}}\left\langle \mathbf{\gamma },M_{k}T_{\frac{%
n}{\beta }}\mathbf{g}\right\rangle _{L^{2}\left( \mathbb{R},\mathbb{C}%
^{N}\right) }e^{2i\pi \frac{nl}{\beta }}=\delta _{k,0}\text{ \ for all }l\in 
\mathbb{Z}\text{,}  \label{superWR}
\end{equation}%
then $\mathcal{G}_{\nu }\left( \mathbf{g},Z_{\beta }\right) $ is a Gabor
superframe for ${L}_{\nu }^{2}\left[ (0,1),\mathbb{C}^{N}\right] $.
\end{proposition}

From this we easily conclude the following.

\begin{proposition}
\label{Criterion}Let $g_{i},\gamma _{i}\in \mathcal{S}_{0}$, $i=1,...,N$. If 
$\mathcal{G}\left( \mathbf{g},\Lambda _{1,\beta }\right) $ is a Gabor
superframe for $L^{2}\left( \mathbb{R},\mathbb{C}^{N}\right) $, then $%
\mathcal{G}_{\nu }\left( \mathbf{g},Z_{\beta }\right) $ is a Gabor
superframe for ${L}_{\nu }^{2}\left[ (0,1),\mathbb{C}^{N}\right] $.
\end{proposition}

\begin{proof}
If $\mathcal{G}\left( \mathbf{g},\Lambda _{1,\beta }\right) $ is a Gabor
superframe for $L^{2}\left( \mathbb{R},\mathbb{C}^{N}\right) $, and $%
g_{i},\gamma _{i}\in \mathcal{S}_{0}$, $i=1,...,N$, then $D_{\mathbf{g}}$
and $D_{\mathbf{\gamma }}$\ are bounded. Thus, the Wexler-Raz
biorthogonality relations \cite[Proposition 2.3]{CharlyYurasuper},%
\begin{equation*}
\beta ^{-1}\left\langle \mathbf{\gamma },M_{k}T_{\frac{n}{\beta }}\mathbf{g}%
\right\rangle _{L^{2}\left( \mathbb{R},\mathbb{C}^{N}\right) }=\delta
_{n,0}\delta _{k,0}\text{ \ for all }k,n\in \mathbb{Z}
\end{equation*}%
hold \cite[Theorem 7.3.1]{Charly}. Since $g_{i},\gamma _{i}\in \mathcal{S}%
_{0}$ then, by Proposition\ \ref{Proposition0}, $\Sigma _{\nu }g_{i}$,$%
\Sigma _{\nu }\gamma _{i}\in {L}_{\nu }^{2}(0,1)$ and\ by Proposition\ (\ref%
{FrameOp1}), $\mathbf{S}_{\mathbf{g,\gamma }}^{\nu }\mathbf{f}$ is bounded.\
This implies (\ref{superWR}) and, by Proposition \ref{Superwr}, $\mathcal{G}%
_{\nu }\left( \mathbf{g},Z_{\beta }\right) $ is a Gabor superframe for ${L}%
_{\nu }^{2}\left[ (0,1),\mathbb{C}^{N}\right] $.
\end{proof}

Now we will use a celebrated result about Gabor superframes with Hermite
functions.

\begin{equation*}
\end{equation*}

\textbf{Theorem B (Gr\"{o}chenig- Lyubarskii }\cite[Theorem 1.1]%
{CharlyYurasuper}\textbf{). }\emph{Let }$h_{N}=\left(
h_{0},...,h_{N-1}\right) $\emph{\ be the vector of the first }$N$\emph{\
Hermite functions. Then }$G\left( \mathbf{h}_{N},\Lambda _{\alpha ,\beta
}\right) $\emph{\ is a Gabor superframe for }$L^{2}\left( \mathbb{R},\mathbb{%
C}^{N}\right) $\emph{\ if and only if }$\alpha \beta <\frac{1}{N}$.

\begin{equation*}
\end{equation*}

By \cite[(45) and Proposition 3.3]{CharlyYurasuper}, $h_{i},\gamma _{i}\in 
\mathcal{S}_{0}$, $i=0,...,N-1$.\ \ Combining Theorem B with Proposition \ref%
{Criterion}, gives the analogue of Theorem B for ${L}_{\nu }^{2}\left[ (0,1),%
\mathbb{C}^{N}\right] $.

\begin{theorem}
If $\beta <\frac{1}{N}$, then $\mathcal{G}_{\nu }\left( \mathbf{h}%
_{N},Z_{\beta }\right) :=\{\Sigma _{\nu }\left( M_{\beta n}\mathbf{h}%
_{N}\right) ,\;n\in \mathbb{Z}\}$ is a \emph{Gabor superframe} for ${L}_{\nu
}^{2}\left[ (0,1),\mathbb{C}^{N}\right] $.
\end{theorem}

Using the unitary super-Bargmann transform (\ref{SuperBarg}), it is a simple
exercise to conclude the following.

\begin{corollary}
If $\beta <\frac{1}{N}$, then $Z_{\beta }$ is a sampling sequence for $%
\mathbf{F}_{\nu }^{(N)}\left( \mathbb{C}/\mathbb{Z}\right) $.
\end{corollary}

\end{document}